\pdfoutput=1
\RequirePackage{ifpdf}
\ifpdf 
\documentclass[pdftex]{sigma}
\else
\documentclass{sigma}
\fi

\numberwithin{equation}{section}

\newtheorem{Theorem}{Theorem}[section]
\newtheorem*{Theorem*}{Theorem}

\newtheorem{Lemma}[Theorem]{Lemma}
\newtheorem{Proposition}[Theorem]{Proposition}

\theoremstyle{definition}
\newtheorem{Definition}[Theorem]{Definition}

\newtheorem{Example}[Theorem]{Example}
\newtheorem{Remark}[Theorem]{Remark}

\newcommand{\Hom}{\operatorname{Hom}}

\newcommand{\Map}{\operatorname{Map}}
\newcommand{\Ad}{\operatorname{Ad}}
\newcommand{\Aut}{\operatorname{Aut}}
\newcommand{\Isom}{\operatorname{Isom}}
\newcommand{\Out}{\operatorname{Out}}

\newcommand{\mon}{\operatorname{mon}}
\newcommand{\Hol}{\operatorname{Hol}}
\newcommand{\Lie}{\operatorname{Lie}}
\newcommand{\curv}{\operatorname{curv}}
\newcommand{\pr}{\operatorname{pr}}
\newcommand{\ave}{\operatorname{ave}}
\newcommand{\bG}{\mathbb{G}}
\newcommand{\bH}{\mathbb{H}}
\newcommand{\bD}{\mathbb{D}}
\newcommand{\bC}{\mathbb{C}}
\newcommand{\bA}{\mathbb{A}}
\newcommand{\bB}{\mathbb{B}}
\newcommand{\bR}{\mathbb{R}}
\newcommand{\bZ}{\mathbb{Z}}
\DeclareMathOperator*{\bigfusion}{\text{\rlap{$\bigoplus$}$\bigotimes$}}

\begin{document}

\allowdisplaybreaks

\newcommand{\arXivNumber}{2512.18380}

\renewcommand{\PaperNumber}{088}

\FirstPageHeading

\ShortArticleName{Finite Group Actions on Quasi-Hamiltonian Spaces}

\ArticleName{Finite Group Actions on Quasi-Hamiltonian Spaces}

\Author{Keito TAKEGOSHI}

\AuthorNameForHeading{K.~Takegoshi}

\Address{Department of Mathematics, Faculty of Science Division I, Tokyo University of Science, \\ 1-3 Kagurazaka, Shinjuku, Tokyo 162-8601, Japan}
\Email{\mail{1125703@ed.tus.ac.jp}}

\ArticleDates{Received December 26, 2025, in final form August 28, 2026; Published online September 18, 2026}

\Abstract{We organize fundamental properties of quasi-Hamiltonian spaces on which a~finite group acts, and we apply them to the theory of moduli spaces of flat connections on an oriented compact surface with boundary.}


\Classification{53D30; 53D20; 58D19}

\section{Introduction}

Moduli spaces of flat connections on an oriented compact surface $\Sigma$ with boundary have attracted much interest from both mathematicians and physicists.
Atiyah and Bott discovered that such a~space carries a~Poisson structure obtained via the infinite-dimensional Hamiltonian reduction for the action of the gauge transformation group on the infinite-dimensional affine space consisting of connections on $\Sigma$ \cite{Atiyah-Bott}.
Since then, several methods for constructing the Atiyah--Bott Poisson structure by finite-dimensional methods (without using the infinite-dimensional Hamiltonian reduction) have been discovered~\cite{AMM,F-R,Goldman,YealKarshon}.

In~\cite{AMM}, Alekseev, Malkin, and Meinrenken developed the theory of quasi-Hamiltonian spaces which are multiplicative analogues of usual Hamiltonian spaces.
In their theory, a~moment map takes values in a~compact Lie group $G$ itself, not in the dual space $\mathfrak{g}^*$ of the Lie algebra of~$G$.
Moreover, using the framework of Dirac geometry, the theory has been extended to the case where $G$ is not necessarily compact~\cite{ABM,BHCM}.
Boalch also developed quasi-Hamiltonian geometry for complex reductive groups whose Lie algebras are equipped with non-degenerate invariant bilinear forms~\cite{Boalch_alg_qHam}.

 Let $\beta \subset \partial\Sigma$ be a~set of base points with exactly one element in each connected component of~$\partial \Sigma$.
Then there exists a~well-known correspondence between flat $G$-bundles on $\Sigma$ framed at~$\beta$ and $G$-valued representations of the fundamental groupoid $\Pi_1(\Sigma,\beta)$ with base points~$\beta$,
$\Hom(\Pi_1(\Sigma,\beta),G)$.
They showed that the representation space $\Hom(\Pi_1(\Sigma,\beta),G)$ becomes a~quasi-Hamiltonian $G^{\beta}$-space where the group-valued moment map is the collection of monodromies around the boundary cycles of $\Sigma$,
and that the Poisson structure on the moduli space is obtained via the quasi-Hamiltonian reduction.

In~\cite{B-Y2015,Meinrenken2017}, the notion of a~quasi-Hamiltonian space was generalized to a~twisted quasi-Hamiltonian space, so that a~moment map could take values in a~bitorsor.
It was shown that the space of $G$-valued twisted representations of the fundamental groupoid (see Section~\ref{section_twisted_reps})
$\Hom_{\kappa}(\Pi_1(\Sigma,\beta),G)$
becomes a~twisted quasi-Hamiltonian $G^{\beta}$-space.
Many examples of twisted representation spaces and their twisted quasi-Hamiltonian structures are discussed in~\cite{Zerouali_Ph.D_Thesis,Zerouali}.

 In this paper, we show that twisted representation spaces associated with $\Sigma$ and a~set of base points containing several points on each boundary component also admit twisted quasi-Hamiltonian structures.
To establish this, we introduce the notion of $\Gamma$-compatible quasi-Hamiltonian spaces and discuss some of their basic properties.

Now, we briefly state the main results of this paper.
In Sections~\ref{2Fin_grp_act_q-Ham_sp} and~\ref{3Basic_properties}, we define $\Gamma$-compatible quasi-Hamiltonian $G$-spaces and study their basic properties, such as the fusion product of such spaces.
Fix a~group homomorphism $\Gamma \to \Aut(G)$. A~$\Gamma$-compatible quasi-Hamiltonian $G$-space is a~quasi-Hamiltonian $G$-space $(M, \omega, \mu)$ where $\Gamma$ acts on $M$, the quasi-Hamiltonian $2$-form $\omega$ is $\Gamma$-invariant, and the group-valued moment map $\mu$ is $\Gamma$-equivariant.
The following theorem is the main result of Sections~\ref{2Fin_grp_act_q-Ham_sp} and~\ref{3Basic_properties}.

\begin{Theorem}[Theorem~\ref{fixed_points_q-Ham}]
 Let $(M, \omega, \mu)$ be a~$\Gamma$-compatible quasi-Hamiltonian $G$-space.
 Let~$M^{\Gamma}$ be the set of $\Gamma$-fixed points and $G^{\Gamma}$ be the Lie subgroup of $\Gamma$-fixed points.
 Also we denote the identity component of $G^{\Gamma}$ by $H$.
 Then, each connected component $N \subset M^{\Gamma}$ together with the $2$-form $\omega|_{N}$ and the map $\mu|_{N}$, which takes values in an~$H$-bitorsor $\bH \subset \bG$ such that~${\mu(N) \subset \bH}$, is a~twisted quasi-Hamiltonian $H$-space.
\end{Theorem}

This theorem allows us to obtain a~new quasi-Hamiltonian space from a~given $\Gamma$-compatible quasi-Hamiltonian space.
In Sections~\ref{4Cyclic_grp_act} and~\ref{5Tw_conn_and_fin_grp_act}, we apply the above results to framed moduli spaces of flat connections.
In addition to the action of the finite group $\Gamma$ on the compact Lie group~$G$, we consider a~free $\Gamma$-action on $\Sigma$ and choose base points $\beta \subset \partial \Sigma$ that are invariant under the $\Gamma$-action on $\Sigma$ and meet each connected component of $\partial \Sigma$.
The following theorem is the main result of Sections~\ref{4Cyclic_grp_act} and~\ref{5Tw_conn_and_fin_grp_act}.

\begin{Theorem}[Theorem~\ref{q-Ham str on rep sp}]
 The representation space $\Hom(\Pi_1(\Sigma,\beta),G)$ has a~structure of a~$\Gamma$-compatible quasi-Hamiltonian $G^{\beta}$-space.
 Moreover, by taking the set of $\Gamma$-fixed points of~$\Hom(\Pi_1(\Sigma,\beta),G)$, we obtain a~structure of a~twisted quasi-Hamiltonian $G^{\beta/\Gamma}$-space on a~twist\-ed representation space \smash{$\Hom_{\mon_{\Sigma/\Gamma, I}}(\Pi_1(\Sigma/\Gamma,\beta/\Gamma),G)$}.
\end{Theorem}

While preparing this paper, we were aware of the work of Meinrenken~\cite{Meinrenken2025}, who showed that the space $\Hom(\Pi_1(\Sigma,\beta),G)$ has a~quasi-Hamiltonian $G^{\beta}$-structure.

The paper is organized as follows.
We give the definition of a~$\Gamma$-compatible quasi-Hamiltonian $G$-space in Section~\ref{2Fin_grp_act_q-Ham_sp}.
We discuss some of its basic properties in Section~\ref{3Basic_properties}.
In Sections~\ref{4Cyclic_grp_act} and~\ref{5Tw_conn_and_fin_grp_act}, we apply the preceding results to moduli spaces of flat connections.

\section{Finite group actions on quasi-Hamiltonian spaces}\label{2Fin_grp_act_q-Ham_sp}
In this section, we give the definition of a~finite group $\Gamma$-compatible $q$-Hamiltonian $G$-space and present a~basic example that is an analogue of the double.

\subsection{Compatibilities of bitorsors}

Let $G$ be a~compact Lie group with Lie algebra $\mathfrak{g}$ and let $\Gamma$ be a~finite group. We fix an action $\Gamma \to \Aut(G)$ of $\Gamma$ on $G$.
Also we give a~$G$-bitorsor $\bG$ equipped with a~left $\Gamma$-action where a~bitorsor $\bG$ is a~manifold endowed with left and right $G$-actions that are both simply transitive and commute with each other.

\begin{Definition} A~$\Gamma$-compatible $G$-bitorsor $\bG$ is a~bitorsor $\bG$ equipped with a~$\Gamma$-action such that $
 \varphi \cdot (g x) = (\varphi \cdot g)(\varphi \cdot x)$, $
 \varphi \cdot (x g) = (\varphi \cdot x)(\varphi \cdot g)$, ($
 \varphi \in \Gamma$, $ g \in G$, $ x \in \bG$).
\end{Definition}

\begin{Example}\label{exam_bitorsor}
 We define the space $\bG := \Map(\bZ/ m \bZ, G)$.
 Fixing an action $\varphi \colon \bZ/m\bZ \to \Aut(G)$, we define a~$\bZ / m \bZ $-action on $\bG$ by
$(l \cdot g)(k) := \varphi(l) g (k - l)$, $ l, k \in \bZ/ m \bZ$, $ g \in \bG$.
 Then the space~$\bG$ is a~$\bZ/ m \bZ$-compatible $\Map(\bZ/ m \bZ, G)$-bitorsor.
\end{Example}

\begin{Remark}\label{compati_bitorsor_and_semi_product}
 Let $\bG$ be a~$\Gamma$-compatible $G$-bitorsor.
 Define left and right actions of~$G \rtimes \Gamma$ on~$\bG \times \Gamma$ by $(g,\varphi)(x,\psi) := (g(\varphi \cdot x), \varphi \psi)$, $(x, \psi)(g,\varphi) := (x(\psi \cdot g), \psi \varphi)$,
 where $(g,\varphi)\in G \rtimes \Gamma$ and~${(x, \psi) \in \bG \times \Gamma}$.
 Then these actions make $\bG \times \Gamma$ into a~$G \rtimes \Gamma$-bitorsor.
 We denote this bitorsor by $\bG \rtimes \Gamma$.
\end{Remark}

We define the inverse $\bG^{-1}$ of a~$G$-bitorsor $\bG$ by the space $\bG^{-1} := \bigl\{x^{-1}\mid x \in \bG\big\}$, which is~a~copy of $\bG$, equipped with the following $G$-actions,
\smash{$ g x^{-1} := \bigl(x g^{-1}\bigr)^{-1}$}, \smash{$ x^{-1} g := \bigl(g^{-1} x\bigr)^{-1}$}
for~${g \in G}$ and $x \in \bG$.

Moreover, we define the product $\bG_1\cdot \bG_2$ of two $G$-bitorsors $\bG_1$ and $\bG_2$ as the quotient of~$\bG_1 \times \bG_2$ by the following action,
$g (x_1,x_2) := \bigl(x_1 g^{-1}, g x_2\bigr)$
for $g \in G$ and $(x_1,x_2) \in \bG_1 \times \bG_2$.
We denote an element $[x_1,x_2] \in \bG_1 \cdot \bG_2$ by $x_1 \cdot x_2$.

\begin{Proposition}
Let $\bG$, $\bG_1$, $\bG_2$ be $\Gamma$-compatible $G$-bitorsors. Suppose that the fixed point sets~$\bG^{\Gamma}$,~$\bG_1^{\Gamma}$,~$\bG_2^{\Gamma}$ are non-empty.
Then we have the following properties:
\begin{enumerate}\itemsep=0pt
 \item [$(1)$] The fixed point set $\bG^{\Gamma}$ becomes a~$G^{\Gamma}$-bitorsor under the restriction of the action of $G$ on~$\bG$.
 \item [$(2)$] The inverse $\bG^{-1}$ of the bitorsor $\bG$ and the product $\bG_1 \cdot \bG_2$ of $\bG_1$ and $\bG_2$ become $\Gamma$-compatible $G$-bitorsors.
 \item [$(3)$] There exists an isomorphism of $G^{\Gamma}$-bitorsor $\bG_1^{\Gamma}\cdot \bG_2^{\Gamma} \simeq (\bG_1 \cdot \bG_2)^{\Gamma}$.
\end{enumerate}
\end{Proposition}

\begin{proof}
 Statement (1) is easy to verify.
 We define actions of $\Gamma$ on \smash{$\bG_1^{-1}$} and $\bG_1\cdot \bG_2$ by \smash{$\varphi x^{-1} := (\varphi x)^{-1}$}, $\varphi (x_1 \cdot x_2) := (\varphi x_1) \cdot (\varphi x_2)$,
 for $\varphi \in \Gamma$ and $x_j \in \bG_j$.
 These actions satisfy the definition of $\Gamma$-compatibility.

 To show statement (3), we see that the canonical map $ \Phi \colon \bG_1^{\Gamma} \cdot \bG_2^{\Gamma} \to (\bG_1 \cdot \bG_2)^{\Gamma}$, $ x_1 \cdot x_2 \mapsto x_1 \cdot x_2$
 is equivariant with respect to the left and right $G^{\Gamma}$-actions.
 For any $g \in G^{\Gamma}$, $x_1 \cdot x_2 \in \bG_1^{\Gamma} \cdot \bG_2^{\Gamma}$,
 \begin{align*}
 \Phi(g (x_1 \cdot x_2)) &= \Phi(g x_1 \cdot x_2) = gx_1 \cdot x_2 = g (x_1 \cdot x_2) = g \Phi(x_1 \cdot x_2).
 \end{align*}
 The equivariance of $\Phi$ with respect to the right $G^{\Gamma}$-action is verified in the same way as above.
 Hence, by the simple transitivity of the left and right actions on bitorsors, $\Phi$ is an isomorphism.\looseness=1
\end{proof}

\subsection[Compatibilities of q-Hamiltonian G-spaces]{Compatibilities of $\boldsymbol{q}$-Hamiltonian $\boldsymbol{G}$-spaces}

First, we recall the definition of twisted quasi-Hamiltonian $G$-spaces~\cite{B-Y2015,Meinrenken2017}.
Throughout this paper, we fix an $\Ad$-invariant positive definite inner product $(\ ,\ )$ on $\mathfrak{g}$ and consider only bitorsors whose corresponding elements in $\Out(G)$ preserve this inner product.
Also we assume that the $\Gamma$-action on $\mathfrak{g}$ induced by the action on $G$ preserves the inner product on $\mathfrak{g}$.

We define the left Maurer--Cartan form $\theta$ and the right Maurer--Cartan form~$\overline{\theta}$ on a~$G$-bitorsor~$\bG$ by
$\theta_x(X) := x^{-1}X$, $ \overline{\theta}_x(X) := Xx^{-1}$, $ x \in \bG$, $ X \in T_x\bG$,
where $x^{-1}X \in \mathfrak{g}$ denotes the derivative of the map $x^{-1} \colon \bG \to G$, $ y \mapsto x^{-1}y$ for $X$.

\begin{Definition} A~(twisted) quasi-Hamiltonian $G$-space $(M, \omega, \mu)$ is a~$G$-manifold $M$ together with an invariant $2$-form $\omega$ and a~$G$-equivariant map $\mu \colon M \to \bG$ to a~$G$-bitorsor $\bG$ such that
\begin{itemize}\itemsep=0pt\setlength{\leftskip}{0.30cm}
 \item[(QH1)] ${\rm d}\omega = \mu^* \frac{1}{12}(\theta,[\theta,\theta])$.
 \item[(QH2)] $\iota \bigl(\xi ^{\#}\bigr)\omega = \frac{1}{2} \mu^* \bigl(\theta + \overline{\theta},\xi\bigr) $, ($\xi \in \mathfrak{g})$, $\xi^{\#}|_p = \frac{\rm d}{{\rm d} t}\bigl({\rm e}^{-t\xi}\cdot p\bigr)|_{t=0}$.
 \item[(QH3)] $\ker \omega_p \cap \ker d\mu_p =\{0\}$, ($p \in M$),
 \end{itemize}
where $\theta$ (resp.\ $\overline{\theta}$) denotes the left (resp.\ right) Maurer--Cartan form on $\bG$.
\end{Definition}

We note that the restriction of the inner product $(\ ,\ )$ to the Lie algebra $\mathfrak{g}^{\Gamma}$ of the $\Gamma$-fixed point Lie subgroup $G^{\Gamma} := \{g \in G\mid \varphi \cdot g = g \text{ for all } \varphi \in \Gamma\}$ is also positive definite and invariant.
Now we give the definition of a~$\Gamma$-compatible quasi-Hamiltonian $G$-space.

\begin{Definition}\label{def_compati_q-Ham}
 A~$\Gamma$-compatible quasi-Hamiltonian $G$-space $(M, \omega, \mu)$ is a~twisted quasi-Hamil\-tonian $G$-space equipped with a~$\Gamma$-action such that
\begin{enumerate}\itemsep=0pt
 \item[$(1)$] $\varphi \cdot (g \cdot p) = (\varphi \cdot g) \cdot (\varphi \cdot p)$, ($\varphi \in \Gamma$, $g \in G$, $ p \in M$).
 \item[$(2)$] The $q$-Hamiltonian $2$-form $\omega$ is $\Gamma$-invariant.
 \item[$(3)$] The moment map $\mu \colon M \to \bG$ is $\Gamma$-equivariant.
 \end{enumerate}
We say that any two $\Gamma$-compatible quasi-Hamiltonian $G$-spaces are isomorphic if there exists a~$\Gamma$-equivariant isomorphism of twisted quasi-Hamiltonian $G$-spaces.
\end{Definition}

From a~$\Gamma$-compatible quasi-Hamiltonian $G$-space, one can naturally construct a~twisted quasi-Hamiltonian $G \rtimes \Gamma$-space.
\begin{Remark}\label{cpti_q-Ham_and_twisted_q-Ham}
 Let $(M,\omega,\mu)$ be a~$\Gamma$-compatible quasi-Hamiltonian $G$-space whose moment map takes values in a~$G$-bitorsor $\bG$.
 The actions of $G$ and $\Gamma$ naturally induce an action of $G \rtimes \Gamma$ on~$M$ given by $(g,\varphi) \cdot p := g \cdot (\varphi \cdot p) $, ($(g,\varphi) \in G \rtimes \Gamma$, $ p \in M$),
 and a~moment map $\tilde{\mu} := M \to \bG \rtimes \Gamma, p \mapsto (\mu(p),1_{\Gamma})$,
 where $\bG \rtimes \Gamma$ is the bitorsor introduced in~Remark~\ref{compati_bitorsor_and_semi_product}.
 Then we obtain a~twisted quasi-Hamiltonian $G \rtimes \Gamma$-space $(M,\omega,\tilde{\mu})$.
\end{Remark}

\subsection[The double D(G\_1,G\_2)]{The double $\boldsymbol{\text{D}(\bG_1,\bG_2)}$}\label{the_double}

We give a~basic example of a~$\Gamma$-compatible quasi-Hamiltonian space.
Let $\bG_1$ and $\bG_2$ be $\Gamma$-compatible $G$-bitorsors.
We define a~$G \times G$-action on $D := \bG_1 \times \bG_2$ by
\[ (g_1, g_2) \cdot (a,b) := \bigl(g_1 a g_2^{-1}, g_2 b g_1^{-1}\bigr), \qquad (g_1,g_2) \in G \times G,\qquad (a,b) \in D.\]
Also we define a~$2$-form $\omega$ and a~map $\mu$ by
\begin{gather*}
 \omega := -\frac{1}{2}\bigl(a^* \theta, b^* \overline{\theta}\bigr)-\frac{1}{2}\bigl(a^*\overline{\theta},b^*\theta\bigr),\\
 \mu \colon \ D \to \bG_1\cdot \bG_2 \times \bG_1^{-1} \cdot \bG_2^{-1},\qquad
(a,b) \mapsto \bigl(ab, a^{-1}b^{-1}\bigr).
\end{gather*}
The tuple $\textbf{D}(\bG_1,\bG_2) := (D,\omega, \mu)$ is a~twisted quasi-Hamiltonian $(G\times G)$-space.
Moreover, under the diagonal action of $\Gamma$ on $D$ and $\bG_1\bG_2 \times \bG_1^{-1}\bG_2^{-1}$, the space $\text{D}(\bG_1,\bG_2)$ carries a~$\Gamma$-compatible structure.

\begin{Remark}\label{the double with different coordinates}
 Let $\bC$ and $\bH$ be $\Gamma$-compatible $G$-bitorsors. We give the double $\textbf{D}(\bC,\bH)$ with different coordinates.
 For any $(C,h) \in \textbf{D}(\bC,\bH)$ and $(g,k) \in G \times G$, we define an action of $G \times G$ on $\textbf{D}(\bC,\bH)$ by $(g,k) \cdot (C,h) := \bigl(kCg^{-1},khk^{-1}\bigr)$.
 We also define a~quasi-Hamiltonian $2$-form by
 \[\omega := \frac{1}{2}\bigl(C^*\overline{\theta},\Ad_{h}C^*\overline{\theta}\bigr)+\frac{1}{2}\bigl(C^*\overline{\theta},h^*\theta+h^*\overline{\theta}\bigr),\]
 and a~moment map by
 \[
 \mu \colon\ \textbf{D}(\bC,\bH) \to \bC^{-1} \bH \bC \times \bH^{-1},\qquad (C,h) \mapsto \bigl(C^{-1}hC, h^{-1}\bigr).
 \]
 Then the twisted quasi-Hamiltonian $(G \times G)$-space $\textbf{D}(\bC,\bH)$ together with the natural $\Gamma$-action is also a~$\Gamma$-compatible quasi-Hamiltonian $(G \times G)$-space and is isomorphic to the double of $\bC^{-1}$ and~$\bH \bC$ with standard coordinates.
\end{Remark}

\section{Basic properties}\label{3Basic_properties}

Twisted quasi-Hamiltonian spaces have been investigated by several authors~\cite{B-Y2015,Knop,Meinrenken2017}.
Although a~$\Gamma$-compatible quasi-Hamiltonian $G$-space naturally gives rise to a~twisted quasi-Hamil\-tonian $(G \rtimes \Gamma)$-space (see Remark~\ref{cpti_q-Ham_and_twisted_q-Ham}), we study some basic properties of $\Gamma$-compatible quasi-Hamiltonian $G$-spaces, since this is the natural framework for the applications to twisted moduli spaces.

\subsection[The set of fixed points of (M, omega, mu)]{The set of fixed points of $\boldsymbol{(M, \omega, \mu)}$}

We fix an action $\Gamma \to \Aut(G)$.
Let $(M,\omega, \mu)$ be a~$\Gamma$-compatible $q$-Hamiltonian $G$-space with the moment map $\mu \colon M \to \bG$.
Let $M^{\Gamma}$ be the fixed point set of $M$ for the action of $\Gamma$ and let $\bG^{\Gamma}$ be the fixed point set of $\bG$.
We denote the identity component of $G^{\Gamma}$ by $H$.
Fix a~connected component $N$ of $M^{\Gamma}$, and we can then take the connected component $\bH$ of $\bG$ so that $\mu(N) \subset \bH$.
Then $\bH$ becomes an~$H$-bitorsor and the following theorem holds.

\begin{Theorem}\label{fixed_points_q-Ham}
 The tuple $(N, \omega|_N, \mu|_N)$ is a~quasi-Hamiltonian $H$-space.
\end{Theorem}

\begin{proof}Let $\theta^{\Gamma}$ (resp.\ \smash{$\overline{\theta}^{\Gamma}$}) be the left (resp.\ right) Maurer--Cartan form on $\bH$.
 We denote the inclusions by $\iota_N \colon N \to M$ and $\iota_{\bH} \colon \bH \to \bG $.
Let the Cartan $3$-forms on $\bG$ and $\bH$ be~\smash{${\chi := \frac{1}{12}(\theta,[\theta,\theta])_{\mathfrak{g}}}$}
 and \smash{$\chi^{\Gamma} := \frac{1}{12}\bigl(\theta^{\Gamma},\bigl[\theta^{\Gamma},\theta^{\Gamma}\bigr]\bigr)_{\mathfrak{g}^{\Gamma}}$},
 respectively.
 Since \smash{$\chi^{\Gamma} = \chi |_{\bH}$}, we obtain the following $
 {\rm d}\omega|_N
 = \bigl(\mu^{*}\chi\bigr)|_{N}
 = (\mu|_N)^* \chi^{\Gamma}$.

Next, we will show the moment map condition.
 For any $\xi \in \mathfrak{g}^{\Gamma}$,
 \begin{align*}
 \iota\bigl(\xi^{\#}\bigr)\omega|_N &= \bigl(\iota\bigl(\xi^{\#}\bigr)\omega\bigr)|_{N} = \left.\left( \frac{1}{2} \mu^*\bigl(\theta + \overline{\theta},\xi \bigr) _{\mathfrak{g}}\right)\right| _{N} = \frac{1}{2}{\mu|_N}^*\bigl(\theta^{\Gamma}+\overline{\theta}^{\Gamma},\xi\bigr)_{\mathfrak{g}^{\Gamma}}.
 \end{align*}

 The degeneracy condition of $\omega|_N$ is shown by using the averaging over $\Gamma$.
 For any $x \in N$ and~${X \in \ker{\omega|_N}_x \cap \ker {\rm d}_{x}\mu|_N}$,
${\rm d}\mu \circ {\rm d}\iota_{N}(X) = {\rm d}\iota_{\bH} \circ {\rm d} \mu|_N(X)= 0$.
 This leads to ${\rm d}\iota_{N}(X) \in \ker {\rm d}_{x}\mu$.
 For any $Y \in T_xM$, the following identity holds from the~$\Gamma$-invariance of~$\omega$
 \begin{align*}
 \omega ({\rm d}\iota_{N}(X), Y) &= \frac{1}{|\Gamma|}\sum_{\varphi \in \Gamma} \varphi^*\omega({\rm d}\iota_{N}(X),Y) = \omega\left({\rm d}\iota_{N}(X), \frac{1}{|\Gamma|}\sum \varphi_*Y\right) =0.
 \end{align*}
Therefore, we obtain ${\rm d}\iota_{N}(X) \in \ker \omega_x \cap \ker {\rm d}_x \mu = \{0\}$.
So, we conclude
\[
\ker\omega|_N \cap \ker {\rm d}\mu|_N =\{0\}.\tag*{\qed}
\]\renewcommand{\qed}{}
\end{proof}

\subsection[Fusion products and G-compatibilities]{Fusion products and $\boldsymbol{\Gamma}$-compatibilities}

Let $G$, $H_1$, and $H_2$ be compact Lie groups.
Let $\bG$, $\bH_1$, and $\bH_2$ be bitorsors for $G$, $H_1$, and $H_2$ respectively.
We fix actions of a~finite group $\Gamma$ on $G$ and $H_j$.
Let $(M_j, \omega_j, \mu_j)$ be a~$\Gamma$-compatible quasi-Hamiltonian $(G\times H_j)$-space for $j=1,2$.
We denote the moment map $\mu_j \colon M_j \to \bG_j \times \bH_j$ by $\mu_j =\bigl(\mu^{G}_j,\mu^{H}_j\bigr)$.

\begin{Proposition}
 The fusion product $M_1 \circledast M_2$ carries a~$\Gamma$-compatible quasi-Hamiltonian $G \times H_1 \times H_2$-structure for the diagonal action of $\Gamma$ on $M_1 \circledast M_2$, $G_1 \times H_1 \times H_2$ and $\bG_1 \bG_2 \times \bH_1 \times \bH_2$.
\end{Proposition}

\begin{proof}
 We show the quasi-Hamiltonian $2$-form on $M_1 \circledast M_2$, $\omega = \omega_1 + \omega_2 - \frac{1}{2}\bigl(\bigl(\mu_1^{G}\bigr)^*\theta, \bigl(\mu_2^{G}\bigr)^* \overline{\theta}\bigr)$
 is $\Gamma$-invariant. Since $\bG_j$ are $\Gamma$-compatible and the inner product $(\ ,\ )$ on $\mathfrak{g}$ is invariant under the action of $\Gamma$, for any $\varphi \in \Gamma$,
 \begin{align*}
 \varphi^*\omega &= \varphi ^* \omega_1 + \varphi ^* \omega_2 -\frac{1}{2}\varphi^*\bigl(\bigl(\mu_1^G\bigr)^*\theta,\bigl(\mu_2^G\bigr)^*\overline{\theta}\bigr) = \omega_1 + \omega_2 -\frac{1}{2}\bigl(\varphi \cdot \bigl(\mu_1^G\bigr)^*\theta, \varphi \cdot \bigl(\mu_2^G\bigr)^* \overline{\theta}\bigr) \\
 &= \omega_1 + \omega_2 -\frac{1}{2}\bigl(\bigl(\mu_1^G\bigr)^*\theta, \bigl(\mu_2^G\bigr)^* \overline{\theta}\bigr) = \omega.
 \end{align*}
 The conditions $(1)$ and $(3)$ in Definition~\ref{def_compati_q-Ham} are easy to verify in this case.
\end{proof}

Moreover, suppose that $M_j^{\Gamma}$ are non-empty connected sets and $\bG_j^{\Gamma}$ and $\bH_j^{\Gamma}$ are non-empty.
One can easily prove the following proposition.

\begin{Proposition}\label{fusion_and_fixedpoints}
 There exists a~canonical twisted quasi-Hamiltonian isomorphism $M_1^{\Gamma} \circledast M_2^{\Gamma} \simeq (M_1 \circledast M_2)^{\Gamma}$.
\end{Proposition}

\begin{Remark}
 In Proposition~\ref{fusion_and_fixedpoints}, if $M_j^{\Gamma}$ are not connected, we can obtain a~similar result as above by taking connected components of $M_1^{\Gamma}$ and $M_2^{\Gamma}$.
\end{Remark}

\begin{Example}[fused double $\bD(\bG_1,\bG_2)$]
 Let $\bG_1$ and $\bG_2$ be $\Gamma$-compatible $G$-bitorsors.
 The internal fusion of the double $\textbf{D}(\bG_1,\bG_2)$ (see Example~\ref{the_double}) is denoted by $\bD(\bG_1,\bG_2)$, and is called the fused double.
 The space $\bD(\bG_1,\bG_2)$ carries a~$\Gamma$-compatible structure.
 Moreover, $\bG_1^{\Gamma}$ and $\bG_2^{\Gamma}$ are submanifolds of $\bG_1$ and $\bG_2$, respectively.
 We obtain an isomorphism of twisted quasi-Hamiltonian $G^{\Gamma}$-space
$\bD(\bG_1,\bG_2)^{\Gamma} \simeq \bD\bigl(\bG_1^{\Gamma},\bG_2^{\Gamma}\bigr)$.
\end{Example}

\section{Cyclic group actions on loop groups}\label{4Cyclic_grp_act}

The results of this section provide the foundation for the constructions in Section~\ref{5Tw_conn_and_fin_grp_act}.
We introduce the notion of Hamiltonian $LG$-spaces compatible with an action of a~cyclic group $\Gamma := \bZ/m\bZ$, examine their relationship with Hamiltonian $L^{(\kappa)}G$-spaces, and discuss the equivalence theorem in the $\Gamma$-compatible setting.
The constructions and computations involving the canonical $2$\nobreakdash-form on~$LG$ follow the same arguments as those in the (twisted) loop group setting developed in~\cite{AMM,LYM_twisted_LG}.

\subsection{Twisted loop groups}\label{Twisted loop groups}
Let $G$ be a~compact connected simply connected Lie group
with Lie algebra $\mathfrak{g}$. We fix an $\Ad$-invariant inner product $(\ ,\ )$ on $\mathfrak{g}$
and a~Lie group automorphism $\kappa \in \Aut(G)$ such that there exists a~positive integer $m$ such that $\kappa^m=1$.
We suppose that $\kappa$ preserves the inner product~$(\ ,\ )$ on $\mathfrak{g}$.
Let $S^{1}_m$ denote $\bR/ m \bZ$.
We define the loop group $LG$ as a~space of maps
\[LG := \Map\bigl(S^1_m,G\bigr)\]
of a~fixed Sobolev class $\lambda > 1/2$.
Also we define the $\kappa$-twisted loop group $L^{(\kappa)}G$ as a~subgroup~of~$LG$
\[L^{(\kappa)}G := \bigl\{g \in LG \mid \kappa g (t) = g(t+1) \text{ for all } t \in S^1_m\bigr\}.\]
We define the spaces $L^{(\kappa)}\mathfrak{g}$ and $L^{(\kappa)}\mathfrak{g}^*$ as subspaces of the space of maps $L\mathfrak{g}:= \Map\bigl(S^1_m, \mathfrak{g}\bigr)$ of~Sobolev class $\lambda$ and the space of $\mathfrak{g}$-valued $1$-forms $L\mathfrak{g}^* := \Omega^1\bigl(S^1_m,\mathfrak{g}\bigr)$ of Sobolev class $\lambda-1$,
\begin{gather*}
 L^{(\kappa)}\mathfrak{g} := \bigl\{\xi \in L\mathfrak{g}\mid \kappa \xi (t) = \xi (t+1) \text{ for all }t \in S^1_m \bigr\}, \\
 L^{(\kappa)}\mathfrak{g}^{*} := \bigl\{A \in L\mathfrak{g}^*\mid \kappa A_t = A_{t+1} \text{ for all }t \in S^1_m\bigr\}.
\end{gather*}
The gauge transformation on $L^{(\kappa)}\mathfrak{g}^*$ is defined by $g \cdot A := gAg^{-1} + {\rm d}gg^{-1}$, $ g \in L^{(\kappa)}G$, $ A \in L^{(\kappa)}\mathfrak{g}^*$.
There exists a~canonical pairing $L^{(\kappa)}\mathfrak{g}^* \times L^{(\kappa)}\mathfrak{g} \to \bR$, \smash{$ (A, \xi) \mapsto \int_{0}^{1}(A,\xi)$}.
We define the holonomy~\smash{$\Hol^{b}_t(A)$} of $A$ starting at $b \in S^{1}_m$ by the unique solution of the following differential equation:
\begin{align*}
 \Hol^{b}_t(A)^{-1} \frac{\partial \Hol^{b}_t(A)}{\partial t} &= -A,\qquad
 \Hol^{b}_b(A) = 1_G.
\end{align*}
Then, the holonomy and the gauge action satisfy the following equation:
\[
\Hol^{b}_t(g\cdot A) = g(b) \Hol ^{b}_t(A) g(t)^{-1}, \qquad g \in L^{(\kappa)}G,\qquad A \in L^{(\kappa)}\mathfrak{g}^*.
\]
Therefore, we have the holonomy map \smash{$\Hol^{(\kappa)} \colon L^{(\kappa)}\mathfrak{g}^{*} \to G^{(\kappa)}$}, \smash{$ A \mapsto \Hol^{0}_1(A)$},
where $G^{(\kappa)} = G$ is a~$G$-bitorsor with the following left and right actions
$g\cdot x := gx$, $x \cdot g := x (\kappa g)$, \big(\smash{$x \in G^{(\kappa)}$}, $ g \in G$\big).

Let $\theta$ (resp.\ $\overline{\theta}$) be the left (resp.\ right) Maurer--Cartan form on $G$.
We define the $2$-form on~$L^{(\kappa)}\mathfrak{g}^*$ by
\[
 \varpi ^{(\kappa)} := \frac{1}{2} \int_{0}^{1}{\rm d}s \left({\Hol^{0}_s}^* \overline{\theta}, \frac{\partial}{\partial s}{\Hol^{0}_s}^* \overline{\theta}\right).
 \]

The following proposition is analogous to~\cite[Proposition 8.1]{AMM}.

\begin{Proposition}We have
\begin{gather*}
 \varpi^{(\kappa)} \text{ is } L^{(\kappa)}G \text{-invariant},\qquad
 {\rm d} \varpi^{(\kappa)} = - {\Hol^{(\kappa)}}^* \frac{1}{12} (\theta,[\theta,\theta]),\\
 \iota\bigl(\xi^{\#}\bigr)\varpi^{(\kappa)} = - {\rm d} \int_{0}^{1}(A,\xi) - \frac{1}{2} {\Hol^{(\kappa)}}^*\bigl(\theta+\overline{\theta},\xi\bigr), \qquad \xi \in L^{(\kappa)}\mathfrak{g}.
\end{gather*}
\end{Proposition}

\begin{proof}
This is shown in the same way as in Proposition~\ref{formula on loop group}.
\end{proof}

\subsection{Cyclic group actions}

In this subsection, let $\Gamma$ be a~cyclic group $\bZ / m \bZ$, and let $S^{1}$ denote $\bR / m \bZ$.
We consider a~$\Gamma$-action on $G$, i.e., a~group homomorphism $\Gamma \to \Aut(G)$. Also a~natural $\Gamma$-action on $S^{1}$ is defined by
$ t \cdot l := t -l $ for $t \in S^{1}$ and $l \in \Gamma$.
Then we can define a~$\Gamma$-action on the loop group~$LG$ by
\[
l \cdot g (t) := \kappa^{l} g(t \cdot l) ,\qquad g \in LG,\qquad t \in S^{1},\qquad l \in \Gamma.
\]
$\Gamma$-actions on $L\mathfrak{g}$ and $L\mathfrak{g}^{*}$ are defined in the same way.
Let $\beta$ be a~finite subset $0\cdot \Gamma = \{0,1,\dots,\allowbreak {m-1}\} \subset S^{1}$, and we denote $G^{\beta} := \Map(\beta,G)$.
Also we denote Maurer--Cartan forms on $G^{\beta}$ by~$\theta$ and~$\overline{\theta}$.
Similar to the twisted case described above, gauge transformations, holonomies and natural pairing are defined.
Furthermore, we define the $2$-form on $L\mathfrak{g}^{*}$ by
\[
\varpi := \frac{1}{2} \sum_{i=1}^{m} \int_{i}^{i+1}{\rm d}s \left({\Hol^{i}_s}^{*}\overline{\theta},\frac{\partial}{\partial s}{\Hol^{i}_s}^{*}\overline{\theta}\right).
\]
Also the holonomy map $\Hol \colon L\mathfrak{g}^{*} \to \bH := G^{\beta}$ is defined by
\smash{$\Hol(A) := \bigl(\Hol^{i}_{i+1}(A)\bigr)^{m}_{i=1}$}, $\bigl(A \in L\mathfrak{g}^{*}\bigr)$,
where $\bH$ is a~$\Gamma$-compatible $G^{\beta}$-bitorsor (Example~\ref{exam_bitorsor}).

In this case, the following proposition holds in the same way as in~\cite[Proposition 8.1]{AMM}.

\begin{Proposition}\label{formula on loop group}We have
 \begin{gather*}
 \varpi\ \text{is}\ LG\text{-invariant}, \qquad
 {\rm d} \varpi = - \Hol ^{*}\frac{1}{12}(\theta,[\theta,\theta]), \\
 \iota\bigl(\xi^{\#}\bigr)\varpi_A = -{\rm d} \oint (A,\xi) - \frac{1}{2}\Hol^*\bigl(\theta + \overline{\theta}, \xi|_\beta\bigr), \qquad \xi \in L\mathfrak{g},\qquad A \in L\mathfrak{g}^*,
 \end{gather*}
 where $\xi|_{\beta} = (\xi(i))_{i \in \beta} \in \mathfrak{g}^{\beta}$.
\end{Proposition}

To prove it, we use the following lemma.

\begin{Lemma}\label{lemma of calc of Hol} We have
 \begin{gather*}
 \iota\bigl(\xi^{\#}\bigr){\Hol^{i}_s}^{*}\overline{\theta} = \Ad\bigl(\Hol^{i}_s\bigr)\xi(s) - \xi(i), \qquad \xi \in L\mathfrak{g}. \\
 \iota(\eta) {\Hol^{i}_s}^*\overline{\theta} = - \int_{i}^{s}{\rm d}u \Ad \bigl(\Hol^{i}_u\bigr)\eta , \qquad \eta \in T_A L\mathfrak{g}^{*}=L\mathfrak{g}^*.
 \end{gather*}
\end{Lemma}

\subsection[Gamma-compatibilities of Hamiltonian LG-spaces]{$\boldsymbol{\Gamma}$-compatibilities of Hamiltonian $\boldsymbol{LG}$-spaces}

In this subsection, we state the definition of a~$\Gamma$-compatibility on Hamiltonian $LG$-spaces.
This is the usual Hamiltonian space version of a~$\Gamma$-compatibility.

First, we recall the definition of a~Hamiltonian $LG$-space~\cite{AMM}.

\begin{Definition} A~Hamiltonian $LG$-space is a~Banach-manifold $N$ together with an $LG$-action, an invariant $2$-form \smash{$\sigma \in \Omega^2(N)^{LG}$}, and an equivariant map \smash{$\Phi \in C^{\infty}\bigl(N, L\mathfrak{g}^*\bigr)^{LG}$} such that
 \begin{itemize}\itemsep=0pt
 \item [(1)] The $2$-form $\sigma$ is closed.
 \item [(2)] The map $\Phi$ is a~moment map for the $LG$-action \smash{$\iota\bigl(\xi^{\#}\bigr)\sigma= -{\rm d} \oint_{S^1}(m, \xi)$}.
 \item [(3)] The $2$-form $\sigma$ is weakly non-degenerate, that is, the induced map \smash{$\sigma^{\flat}_x \colon T_xN \to T^{*}_xN$} is injective.
 \end{itemize}
\end{Definition}

\begin{Definition}
Let $(N,\sigma,\Phi)$ be a~Hamiltonian $LG$-space and we fix a~$\Gamma$-action on $N$.
 The tuple $(N,\sigma,\Phi)$ together with the $\Gamma$-action is called a~$\Gamma$-compatible Hamiltonian $LG$-space if it satisfies the following properties:
\begin{itemize}\itemsep=0pt
 \item[(1)] $l \cdot (g \cdot x) = (l \cdot g) \cdot (l \cdot x) $, $l \in \Gamma$, $ g \in LG$, $ x \in N$.
 \item[(2)] $\sigma$ is $\Gamma$-invariant.
 \item[(3)] $\Phi \colon N \to L\mathfrak{g}^*$ is $\Gamma$-equivariant.
 \end{itemize}

 We say that two $\Gamma$-compatible Hamiltonian $LG$-spaces are isomorphic
 if there exists a~$\Gamma$\nobreakdash-equi\-variant isomorphism of Hamiltonian $LG$-spaces.
\end{Definition}

\subsection[Fixed points of a Gamma-compatible space]{Fixed points of a~$\boldsymbol{\Gamma}$-compatible space}

Let $(N,\sigma,\Phi)$ be a~$\Gamma$-compatible Hamiltonian $LG$-space.
Then, we obtain a~Hamiltonian $L^{(\kappa)}G$-space from a~$\Gamma$-compatible Hamiltonian $LG$-space by taking fixed point objects.
Now $\Gamma$-fixed objects are defined by
\begin{itemize}\itemsep=0pt
 \item $N^{\Gamma} := \{x \in N \mid l \cdot x = x,\ l \in \Gamma\}$,
 \item $\sigma^{\Gamma} := \sigma|_{N^{\Gamma}}$,
 \item $\Phi ^{\Gamma} \colon N^{\Gamma} \to L\mathfrak{g}^{*}$, $x \mapsto \Phi(x)$.
\end{itemize}
In this paper, we denote $\Gamma$-fixed objects as above. For example, the $\Gamma$-fixed subgroup of~$LG$ is denoted by $LG^{\Gamma}$.
Each connected component of $N^{\Gamma}$ is a~Banach submanifold of $N$.
By considering left transformation of $LG^{\Gamma}$,
the subgroup $LG^{\Gamma}$ is a~Banach Lie subgroup of~$LG$ (see Appendix~\ref{Appendix C}).

We thus obtain the following proposition.

\begin{Proposition}
 For any connected component $L$ of $N^{\Gamma}$, the tuple
 $\bigl(L, \sigma^{\Gamma}|_L, \Phi^{\Gamma}|_L\bigr)$ is a~Hamiltonian $LG^{\Gamma}$-space, i.e., the following properties hold:
\begin{itemize}\itemsep=0pt
 \item $LG^{\Gamma}$ acts on $L$.
 \item $\sigma^{\Gamma}|_L$ is a~$LG^{\Gamma}$-invariant symplectic form on~$L$.
 \item $\Phi^{\Gamma}|_L$ is $LG^{\Gamma}$-equivariant and $\sigma^{\Gamma}|_L$, $\Phi^{\Gamma}|_L$ satisfy the moment map condition
 \[\iota\bigl(\xi^{\#}\bigr)\sigma^{\Gamma} = - {\rm d} \oint \bigl(\Phi^{\Gamma},\xi\bigr), \qquad \xi \in L\mathfrak{g}^{\Gamma}.\]
 \end{itemize}
\end{Proposition}

\begin{proof}
 Since $G$ is simply connected, the space $LG^{\Gamma}$ is connected.
 Therefore, $LG^{\Gamma}$ acts on a~connected component $L$.
 The moment map condition for $\Phi^{\Gamma}|_L$ is a straightforward calculation and
 the weak non-degeneracy of $\sigma^{\Gamma}$ is proved by averaging over $\Gamma$.
 Let $\iota \colon N^{\Gamma} \to N$ be the inclusion map.
 For any $\xi \in L\mathfrak{g}^{\Gamma}$,
 \begin{align*}
 \iota\bigl(\xi^{\#}\bigr)\sigma^{\Gamma}|_L &= \bigl(\iota\bigl(\xi^{\#}\bigr)\sigma\bigr)|_{L} = \left. -{\rm d} \oint (\Phi,\xi) \right|_{L} = - \oint \bigl(\Phi^{\Gamma} |_{L},\xi \bigr).
 \end{align*}
 Next, we will show the non-degeneracy of $\sigma^{\Gamma}$.
 Fix an element $x \in N^{\Gamma}$ and $X \in \ker\sigma^{\Gamma}$.
 Then for any $Y \in T_xN$,
 \begin{align*}
 \iota(Y)\iota(X) \sigma &= \sigma(X,Y)= \frac{1}{|\Gamma|}\sum_{l \in \Gamma} l^*\sigma(X,Y) = \sigma\left( X,\ \frac{1}{|\Gamma|} \sum_{l \in \Gamma}l_*Y \right)
 = 0.
 \end{align*}
 Thus we have proved the non-degeneracy condition.
\end{proof}

\begin{Remark}Suppose that $\Gamma$ is a~cyclic group $\bZ/m\bZ$ and that the action of $\Gamma$ on $LG$ is the one defined in Section~\ref{Twisted loop groups}.
 Then we can identify $LG^{\Gamma}$ with the twisted loop group $L^{(\kappa)}G$,
\begin{align*}
 LG^{\Gamma} &= \{g \in LG\mid l\cdot g = g,\, l \in \Gamma\} = \bigl\{g \in LG\mid \kappa g(t) =g(t+1) \text{ for all }t \in S^1_m\bigr\}
 = L^{(\kappa)}G.
\end{align*}
For any $\xi \in L^{(\kappa)}\mathfrak{g} = L\mathfrak{g}^{\Gamma}$,
\begin{align*}
 \iota\bigl(\xi^{\#}\bigr)\sigma^{\Gamma}
 = -{\rm d} \oint \bigl(\Phi^{\Gamma}, \xi\bigr) = -{\rm d} \sum_{i=1}^{m} \int_{i}^{i+1} \bigl(\Phi^{\Gamma}, \xi\bigr)
 = -m{\rm d} \int_{0}^{1} (\Phi,\xi).
\end{align*}
Therefore, $\big(L, \frac{1}{m}\sigma^{\Gamma}, \Phi^{(\kappa)}\big)$ becomes a~Hamiltonian $L^{(\kappa)}G$-space.
\end{Remark}

\subsection{Equivalence theorem}
It is known that there exists a~natural correspondence between (twisted) Hamiltonian $LG$-spaces and (twisted) quasi-Hamiltonian $G$-spaces
which is called the equivalence theorem~\cite{AMM,LYM_twisted_LG,Meinrenken2017}.
Taking base points $\beta$ with $\beta = 0 \cdot \Gamma \subset S^{1}_m$, the equivalence theorem between $\Gamma$-compatible Hamiltonian $LG$-spaces and $\Gamma$-compatible quasi-Hamiltonian $G^{\beta}$-spaces can be shown.
Let $\Omega^{\beta}G$ be a~based loop group for $\beta $, i.e., $\Omega ^{\beta}G = \{g \in LG\mid g|_{\beta} = 1_G\}$.

\begin{Theorem}\label{equiv thm}
Taking the holonomy manifold for $\Omega^{\beta}G$ yields a~bijective correspondence between
\begin{itemize}\itemsep=0pt
 \item the isomorphism class of $\Gamma$-compatible Hamiltonian $LG$-spaces $(N,\sigma,\Phi)$ with a~proper moment map,
 \item the isomorphism class of $\Gamma$-compatible $q$-Hamiltonian $G^{\beta}$-spaces $(M,\omega,\mu)$.
\end{itemize}
\end{Theorem}

\section{Twisted connections and finite group actions}\label{5Tw_conn_and_fin_grp_act}

\subsection{Preliminaries}
Let $\Sigma$ be a~compact connected 2-manifold with boundary components $V_0, V_1, \dots, V_{r}$ and $P := \Sigma \times G$ a~trivial bundle on $\Sigma$.
Fix $\lambda > 1$ and let $\Omega^{j}(\Sigma,\mathfrak{g})$ be the space of $\mathfrak{g}$-valued j-forms of~Sobolev class~${\lambda-j}$.
Then the gauge $\mathcal{G}(\Sigma) := \Map(\Sigma,G)$ is a~Banach Lie group modeled on $\Lie(\mathcal{G}(\Sigma)) = \Omega^{0}(\Sigma,\mathfrak{g})$.
Let $\mathcal{A}(\Sigma) := \Omega^{1}(\Sigma,\mathfrak{g})$ be the space of connections on the trivial bundle on $\Sigma$.
$\mathcal{A}(\Sigma)$ has a~smooth $\mathcal{G}(\Sigma)$-action given by
\smash{$ g\cdot A = g A g^{-1} + g^{*}\overline{\theta}$}
The space $\mathcal{A}(\Sigma)$ has a~canonical symplectic form called Atiyah--Bott $2$-form
$\sigma = \frac{1}{2} \int_{\Sigma} (\delta A\wedge \delta A )$
for which $\mathcal{G}(\Sigma)$-action is Hamiltonian with moment map
\[\langle \Psi(A), \xi\rangle = \int_{\Sigma} (\curv A,\xi) - \int _{\partial \Sigma} (A,\xi), \qquad \xi \in \Omega^{0}(\Sigma,\mathfrak{g}),
\]
where $\curv(A) $ is the curvature of $A$.
The space of flat connections is denoted by $\mathcal{A}_{\mathrm{flat}}(\Sigma)$.

\subsection{Finite group actions}
Let $\Gamma$ be a~finite group.
Fix a~free right $\Gamma$-action on $\Sigma$ which preserves the orientation of $\Sigma$, and a~left $\Gamma$-action on G which preserves the inner product on $\mathfrak{g}$.
Then a~right $\Gamma$-action on the trivial bundle is defined by
\[(x,a) \cdot \varphi = \bigl(x\cdot \varphi, \varphi ^{-1} \cdot a\bigr) ,\qquad (x,a) \in P,\qquad \varphi \in \Gamma.\]
This action satisfies $\Gamma$-compatibility
\[ (p\cdot g)\cdot \varphi = (p \cdot \varphi)\cdot (g \cdot \varphi), \qquad p \in P,\qquad g \in G,\qquad \varphi \in \Gamma,
\]
where $g\cdot \varphi = \varphi^{-1} \cdot g$.
We can also define left $\Gamma$-actions on $\Omega^{j}(P,\mathfrak{g})$ by
\[
\varphi \cdot \alpha = \varphi \circ \alpha \circ d\varphi , \qquad \alpha \in \Omega^{j}(P,\mathfrak{g}),\qquad \varphi \in \Gamma,
\]
and left $\Gamma$-actions on $\Omega^{j}(\Sigma,\mathfrak{g})$ can be defined in the same way.

\begin{Lemma}\label{lemm of compati for conn}
 The canonical $2$-form $\sigma$ is $\Gamma$-invariant and the curvature $\curv \colon \mathcal{A}(\Sigma) \to \Omega^{2}(\Sigma,\mathfrak{g})$ is $\Gamma$-equivariant.
 Moreover, the following $\Gamma$-compatibility holds:
 \[
 \varphi \cdot (g \cdot A) = (\varphi \cdot g) \cdot (\varphi \cdot A), \qquad \varphi \in \Gamma,\qquad g \in \mathcal{G}(\Sigma),\qquad A \in \mathcal{A}(\Sigma).
 \]
\end{Lemma}

\begin{proof}For any $\varphi \in \Gamma$ and $X, Y \in T_A\mathcal{A}(\Sigma) = \Omega^{1}(\Sigma,\mathfrak{g})$, the following holds:
 \[\varphi^*\sigma(X,Y) = \int_{\Sigma} (\varphi X \circ {\rm d} \varphi \wedge \varphi Y \circ {\rm d} \varphi) = \int_{\Sigma}(X\wedge Y) = \sigma(X,Y).\]

Since the exterior derivative ${\rm d} \colon \Omega^{j}(\Sigma,\mathfrak{g}) \to \Omega^{j+1}(\Sigma,\mathfrak{g})$ is $\Gamma$-equivariant, the $\Gamma$-equivariance of the curvature holds
 \begin{align*}
 \curv(\varphi \cdot A) &= {\rm d}(\varphi \cdot A) + \frac{1}{2}[ \varphi \cdot A \wedge \varphi \cdot A] = \varphi \cdot {\rm d} A + \frac{1}{2}\varphi \cdot [A \wedge A]
 = \varphi \cdot \curv A,
 \end{align*}
 where $\varphi \in \Gamma$ and $A \in \mathcal{A}(\Sigma)$.
 Moreover, for any $g \in \mathcal{G}(\Sigma)$, $\varphi \in \Gamma$ and $A \in \mathcal{A}(\Sigma)$, the following holds
$\varphi \circ \Ad(g) = \Ad(\varphi \cdot g) \circ \varphi$,
$\varphi \cdot g^*\overline{\theta} = (\varphi \cdot g)^* \overline{\theta}$.
 Therefore, we obtain the $\Gamma$-compatibility of~$\mathcal{G}(\Sigma)$ and $\mathcal{A}(\Sigma)$.
\end{proof}

$\Gamma$-actions on $\mathcal{G}(\partial \Sigma)$ and $\Omega^{j}(\partial \Sigma, \mathfrak{g})$ are also defined in the same way.

\subsection[The space M(Sigma) and Gamma-compatibilities]{The space $\boldsymbol{\mathcal{M}(\Sigma)}$ and $\boldsymbol{\Gamma}$-compatibilities}

Since the restriction map $\mathcal{G}(\Sigma) \to \mathcal{G}(\partial \Sigma)$ is surjective,
the following sequence is exact:
\[ 1 \to \mathcal{G}_{\partial \Sigma}(\Sigma) \to \mathcal{G}(\Sigma) \to \mathcal{G}(\partial \Sigma) \to 1,
\]
where $\mathcal{G}_{\partial \Sigma}(\Sigma) := \{g \in \mathcal{G}(\Sigma)\mid g|_{\partial \Sigma} =1_G\}$.
Then $\mathcal{M}(\Sigma) := \mathcal{A}_{\mathrm{flat}}(\Sigma) / \mathcal{G}_{\partial \Sigma}(\Sigma)$ is a~Banach manifold and a~Hamiltonian $(\mathcal{G}(\partial \Sigma) \simeq LG^{r+1})$-space with proper moment map~\cite{AMM,Donaldson1992}.
The symplectic form~$\sigma$ on $\mathcal{M}(\Sigma)$ is induced from Atiyah--Bott $2$-form on $\mathcal{A}(\Sigma)$ and the moment map is a~restriction
\[
\Phi \colon \mathcal{M}(\Sigma) \to \Omega^1(\partial \Sigma, \mathfrak{g}),\qquad A \mapsto A|_{\partial \Sigma}.
\]
The $\Gamma$-compatibility of $\mathcal{M}(\Sigma)$ can be shown by using Lemma~\ref{lemm of compati for conn}. A~$\Gamma$-action on $\mathcal{M}(\Sigma)$ is defined~by
$\varphi \cdot [A] := [\varphi \cdot A]$, $ \varphi \in \Gamma$, $ [A] \in \mathcal{M}(\Sigma)$.
The Hamiltonian $\mathcal{G}(\partial \Sigma)$-space $(\mathcal{M}(\Sigma), \sigma, \Phi)$ together with the $\Gamma$-action is a~$\Gamma$-compatible Hamiltonian $\mathcal{G}(\partial \Sigma)$-space.

\begin{Proposition}
 $(\mathcal{M}(\Sigma),\sigma,\Phi)$ is a~$\Gamma$-compatible Hamiltonian $\mathcal{G}(\partial \Sigma)$-space.
\end{Proposition}
The above proposition follows immediately from Lemma~\ref{lemm of compati for conn}.

\subsection[Fixed points of M(Sigma)]{Fixed points of $\boldsymbol{\mathcal{M}(\Sigma)}$}

By taking a~fixed point objects $\bigl(\mathcal{M}(\Sigma)^{\Gamma},\omega^{\Gamma},\Phi^{\Gamma}\bigr)$, we get Hamiltonian spaces for twisted loop groups actions.

We define a~$\Gamma$-action on $I := \{ 0,1,\dots,r\}$ by using a~permutation of $\partial \Sigma = \coprod _{i=0}^{r} V_i$,
$V_{i \cdot \varphi } = V_{i} \cdot \varphi$.
Let $\Gamma(i) $ be a~stabilizer of $i \in I$, i.e., $\Gamma(i) = \{\varphi \in \Gamma\mid i \cdot \varphi = i\}$.

\begin{Lemma}
 For any $i \in I$, $\Gamma(i)$ is a~cyclic group.
\end{Lemma}

\begin{proof}
 Fix a~point $b \in V_i$ and consider a~$\Gamma(i)$-orbit $b \cdot \Gamma(i)$, we restrict the $\Gamma(i)$-action on $\partial \Sigma$ to~$b \cdot \Gamma(i)$.
 Since this $\Gamma$-action on $\partial \Sigma$ is free, the restricted $\Gamma(i) $-action on $b \cdot \Gamma(i)$ is simply transitive.
 So by fixing a~cyclic order of $b \cdot \Gamma(i) =\{b_0, b_1, \dots, b_{m_i-1}\}$, we get a~unique element~$b_1 = b_0 \cdot \kappa_i$.
 Then, the element $\kappa_i$ is a~generator of $\Gamma(i) = \langle \kappa_i \rangle$.
\end{proof}

We fix a~generator $\kappa_i$ of $\Gamma(i)$ for all $i \in I$.
We set $m_i := | \Gamma(i)|$, and identify $V_i$ with $S^{1}_{m_i} = \bR / m_i \bZ$ while keeping $\Gamma(i)$-action.
So, the following isomorphism is obtained:
\smash{$
 \mathcal{G}(V_i)^{\Gamma(i)} \simeq L^{(\kappa_i)}G$}.
So, by taking a~complete system of representatives $I^{\prime}$ of $I/\Gamma$, and fixing $\kappa_i$ for any $i \in I$,
and taking~${\tau_{ji} \in \Gamma }$ such that \smash{$\kappa_j = \tau_{ji} \kappa_i \tau_{ji}^{-1}$} for $i, j \in i\cdot \Gamma$,
we identify the fixed points subgroup~$\mathcal{G}(\partial \Sigma)^{\Gamma}$ with a~product of twisted loop groups,
\begin{align*}
 \mathcal{G}(\partial \Sigma)^{\Gamma} &\simeq \left\{(g_i) \in \underset{i \in I}{\prod} L^{(\kappa_i)}G\,\bigg|\, g_i = \tau_{ij} g_j\right\} \simeq \underset{i \in I^{\prime}}{\prod} L^{(\kappa_i)}G.
\end{align*}
In the same way, we obtain the following isomorphisms:
\begin{align*}
 \Omega^{1}(\partial \Sigma, \mathfrak{g})^{\Gamma} &\simeq \underset{i \in I^{\prime}}{\prod} L^{(\kappa_i)}\mathfrak{g}^{*};\qquad A \mapsto (A|_{V_i}), \qquad
 \Omega^{0}(\partial \Sigma, \mathfrak{g})^{\Gamma} \simeq \underset{i \in I^{\prime}}{\prod} L^{(\kappa_i)}\mathfrak{g};\qquad \xi \mapsto (\xi|_{V_i}).
\end{align*}
We define $\Phi^{\prime}$ and $\sigma^{\prime}$ by
\smash{$
 \Phi^{\prime} \colon \mathcal{M}(\Sigma)^{\Gamma} \to {\prod}_{i \in I^{\prime}} L^{(\kappa_i)}\mathfrak{g}^{*}$}; $ A \mapsto (\Phi(A)|_{V_i})$, \smash{$
 \sigma^{\prime} := \frac{1}{|\Gamma |} \sigma$}.
By considering the \smash{$\prod_{i \in I} L^{(\kappa_i)}G$}-action on $\mathcal{M}(X)$, the following proposition holds.

\begin{Proposition}
 Each connected component $\mathcal{N}$ of $\mathcal{M}(\Sigma)^{\Gamma}$ together with the $2$-form $\sigma^{\prime}|_{\mathcal{N}}$ and the moment map $\Phi^{\prime}|_{\mathcal{N}}$ is a~Hamiltonian \smash{$\Pi_{i \in I^{\prime}} L^{(\kappa_i)}G$}-space.
\end{Proposition}

\begin{proof}
 It is trivial that $\sigma^{\prime}$ is symplectic and $\Phi^{\prime}$ is equivariant for twisted loop group actions.
 So, we prove the moment map condition:
 for any $[A] \in \mathcal{M}(\Sigma)$ and \smash{$(\xi_i) = \xi \in \prod_{i \in I^{\prime}} L^{(\kappa_i)}\mathfrak{g}$},
\begin{align*}
 \iota\bigl(\xi^{\#}\bigr)\sigma^{\prime}_{[A]} &= \frac{1}{| \Gamma|} \iota\bigl(\xi^{\#}\bigr)\sigma_{[A]}|_{\Gamma}
 = -\frac{1}{| \Gamma|}{\rm d} \int_{\partial \Sigma} (\Phi(A), \xi)
 = -\frac{1}{| \Gamma|} \sum_{i \in I} {\rm d} \oint_{V_i}(A|_{V_i}, \xi|_{V_i}) \\
 &= -\frac{1}{| \Gamma|} \sum_{i \in I^{\prime}} \left| i \cdot \Gamma\right| {\rm d} \oint_{V_i}(A|_{V_i}, \xi_i)
 = -\frac{1}{| \Gamma|} \sum_{i \in I^{\prime}} \left| i \cdot \Gamma\right| \left| \Gamma(i)\right| {\rm d} \int_{0}^{1} (A|_{V_i}, \xi_i) \\
 &
 = - \sum_{i \in I^{\prime}} {\rm d}\int_{0}^{1} \bigl(\Phi^{\prime}(A)_i, \xi_i\bigr).\tag*{\qed}
\end{align*} \renewcommand{\qed}{}
\end{proof}

\subsection[The space M(Sigma) and Gamma-compatibilities]{The space $\boldsymbol{M(\Sigma)}$ and $\boldsymbol{\Gamma}$-compatibilities}

Let $\beta \subset \partial \Sigma$ be a~finite subset (base points) which is $\Gamma$-invariant and we fix a~cyclic ordering of~$\beta_i = \bigl(b^{i}_{\lambda}\bigr) := \beta \cap V_i$.
We denote
$\mathcal{G}_{\beta}(\Sigma) := \{g \in \mathcal{G}(\Sigma)\mid g|_{\beta} = 1\}$, $
 \mathcal{G}_{\beta}(\partial \Sigma) := \{g \in \mathcal{G}(\partial \Sigma)\mid g|_{\beta}=1\}$.
Set $M(\Sigma) := \mathcal{A}_{\mathrm{flat}}(\Sigma)/\mathcal{G}_{\beta}(\Sigma) \simeq \mathcal{M}(\Sigma)/\mathcal{G}_{\beta}(\partial \Sigma)$.
An action of $\mathcal{G}(\partial \Sigma)/\mathcal{G}_{\beta}(\partial \Sigma) \simeq G^{\beta}$ on $M(\Sigma)$ is induced from the residual gauge action on $\mathcal{M}(\Sigma)$.
The collection of holonomy map \smash{$\Hol^{b^i_{\lambda}}_{b^i_{\lambda+1}}(A)$} defines a~$G^{\beta}$-equivariant map
\[
\mu \colon\ M(\Sigma) \to \bH,\qquad A \mapsto \Hol(A) := \bigl(\Hol^{b^i_{\lambda}}_{b^i_{\lambda+1}}(A)\bigr)_{b^i_{\lambda} \in \beta},
\]
where $\bH = G^{\beta}$ is a~$G^{\beta}$-bitorsor under the following left and right actions,
$(g \cdot h)\bigl(b^{i}_{\lambda}\bigr) = g\bigl(b^i_{\lambda}\bigr)h\bigl(b^i_{\lambda}\bigr)$, $
 (h \cdot g)\bigl(b^i_{\lambda}\bigr) = h\bigl(b^i_{\lambda}\bigr)g\bigl(b^{i}_{\lambda+1}\bigr)$, $ g \in G^{\beta}$, $ h \in \bH$.
An action of $\Gamma$ on $M(\Sigma)$ is induced from the $\Gamma$-action on $\mathcal{M}(\Sigma)$.
The following theorem, which is analogous to~\cite[Theorem 9.1]{AMM}, holds.

\begin{Theorem}\label{q-Ham str on M(X)}
 There is a~natural $G^{\beta}$-invariant $2$-form $\omega$ on $M(\Sigma)$ for which $(M(\Sigma),\omega,\mu)$ is a~$\Gamma$-compatible $q$-Hamiltonian $G^{\beta}$-space.
\end{Theorem}

\begin{proof}
 The proof is analogous to that of~\cite[Theorem 9.1]{AMM}.
 In this proof, let $\pi \colon \mathcal{M}(\Sigma) \to M(\Sigma) $ be the quotient map.
 First, we will show the equivariance of $\mu \colon M(\Sigma) \to \bH$.
 For any $g \in G^{\beta}$ and $\pi(A) \in M(\Sigma)$,
$\mu(g \cdot \pi(A)) = \mu (\pi(\tilde{g}\cdot A)) = \Hol(\tilde{g}\cdot A)= \Ad_{g} \mu(\pi(A))$.

Let $R_{j} \colon \Omega^1(\partial \Sigma, \mathfrak{g}) \to \Omega^1(V_j,\mathfrak{g})$ be a~restriction and let $\varpi_{j}$ be the canonical $2$-form on $\Omega^{1}(V_j, \mathfrak{g}) \simeq L\mathfrak{g}^*$.
 Denote $\varpi := \sum_{j=0}^{r} R_{j}^*\varpi_{j}$. We define a~$2$-form on $\mathcal{M}(\Sigma)$, $\tilde{\omega} := \sigma - \Phi^*\varpi$.
 Then the $2$-form $\tilde{\omega}$ is basic for $\mathcal{G}_{\beta}(\partial \Sigma)$-action for $\mathcal{M}(\Sigma)$.
 Since $\sigma$ and $\varpi$ are $\mathcal{G}(\partial \Sigma)$-invariant, $\tilde{\omega}$~becomes a~$\mathcal{G}(\partial \Sigma)$-invariant $2$-form.
 Furthermore, for any $\xi \in \Lie(\mathcal{G}_{\beta}(\partial \Sigma))$,
 \begin{align*}
 \iota\bigl(\xi^{\#}\bigr)\tilde{\omega}_A &= \iota\bigl(\xi^{\#}\bigr)\sigma_A - \iota\bigl(\xi^{\#}\bigr) \Phi^* \varpi_A
 = \iota\bigl(\xi^{\#}\bigr) \sigma_A - \sum_{j=0}^{r} {R_j}^* \bigl(\iota\bigl({\xi|_{V_j}}^{\#}\bigr)\varpi_j\bigr)\\
 &= -{\rm d} \int_{\partial \Sigma }(A|_{\partial \Sigma}, \xi)- \sum_{j=0}^{r} {R_{j}}^*\left(-{\rm d} \oint (A|_{V_j},\xi|_{V_j}) -\frac{1}{2}(\Hol_{j})^* \bigl(\theta_{j}+ \overline{\theta}_{j}, \xi|_{\beta_{j}}\bigr) \right)= 0,
 \end{align*}
 where $\Hol_j \colon \Omega^1(V_{j},\mathfrak{g}) \to \bH^{j} := G^{\beta_j}$, \smash{$A_j \mapsto \bigl(\Hol^{b^{j}_\lambda}_{b^{j}_\lambda}(A_j)\bigr)$}, and $\theta_j $ (resp.\ $\overline{\theta}$) is the left (resp.\ right) Maurer--Cartan form on~$\bH^j$.
 So, $\tilde{\omega}$ induces a~unique $G^{\beta}$-invariant $2$-form $\omega$ on $M(\Sigma)$ which satisfies $\pi^* \omega = \tilde{\omega}$.

 Now we prove the tuple $(M(\Sigma),\omega,\mu)$ is a~twisted quasi-Hamiltonian $G^{\beta}$-space.
 We denote the canonical $3$-form on $\bH$ by $\chi_{\bH}$, i.e., $\chi_{\bH} = \frac{1}{12}(\theta_{\bH},[\theta_{\bH},\theta_{\bH}])$.
 Then, by $\Hol \circ \Phi = \mu \circ \pi$,
$
 \pi^*{\rm d}\omega = {\rm d}\tilde{\omega}= {\rm d}( \sigma - \Phi^* \varpi)
 = \Phi^* \Hol^* \chi_{\bH}
 = \pi^* \mu^* \chi_{\bH}$.
 Therefore, we have proved ${\rm d} \omega = \mu^* \chi_{\bH}$.

 Next, we want to show the moment map condition.
 For any $\xi \in \mathfrak{g}^{\beta}$, there exists $\tilde{\xi} \in \Omega^0(\partial \Sigma,\mathfrak{g})$ such that $\tilde{\xi}|_{\beta} = \xi$.
 Then
 \begin{align*}
 \pi^*\bigl(\iota\bigl(\xi^{\#}\bigr)\omega\bigr) &= \iota\bigl(\tilde{\xi}\bigr) \pi^*\omega
 = \iota\bigl(\tilde{\xi}^{\#}\bigr) \bigl(\sigma-\Phi^*\varpi\bigr)
 = \frac{1}{2} \sum_{j=0}^{r} {R_j}^*{\Hol_j}^*\bigl(\theta_j+ \overline{\theta}_j,\tilde{\xi}|_{\beta_{j}}\bigr)\\
 &= \pi^* \left( \frac{1}{2} \mu^* \bigl(\theta_{\bH} + \overline{\theta}_{\bH},\xi\bigr)\right).
 \end{align*}
 So, the moment map condition holds.

 Finally, we will prove the degeneracy of $\omega$, i.e., we will show $\ker \omega_m \cap \ker d_m \mu = \{0\}$ for any~$m \in M(\Sigma)$.
 Fix $v \in \ker \omega_m \cap \ker {\rm d}_m \mu$ and $A \in \mathcal{M}(\Sigma)$ such that $\pi(A)=m$.
 Since $\pi \colon \mathcal{M}(\Sigma) \to M(\Sigma)$ is a~submersion, there exists a~tangent vector $X \in T_A \mathcal{M}(\Sigma)$ such that ${\rm d}_A\pi(X)=v$.

 For any $Y \in T_A\mathcal{M}(\Sigma)$,
$\tilde{\omega}_A(X,Y) = \pi^*\omega _A(X,Y) = 0$.
 Therefore, $X \in \ker \tilde{\omega}_A$. By the definition of $\tilde{\omega}$, the following holds:
$\iota(X) \sigma= \iota(X) \Phi^*\varpi$.
Thus, $X \in (\ker {\rm d}_A \Phi )^{\sigma}$.
 Since the space $( \ker {\rm d}_A \Phi )^{\sigma}$ equals to the tangent space of the orbit $\mathcal{G}(\partial \Sigma) \cdot A$, i.e., there exists $\xi \in \Omega^0(\partial \Sigma, \mathfrak{g})$ that the generating vector $\xi^{\#}$ equals to~$X$ \cite[Theorem~9.1]{AMM}.

 Since \smash{${\xi|_{\beta}}^{\#} \in \ker {\rm d}_m \mu$}, we obtain
 $\xi|_{\beta} \in \ker(\Ad_{\mu(m)}-1)$.
 Now we define $\eta_j \in \Omega^0(V_j,\mathfrak{g})$ by~\smash{$
 \eta_j(s) := \Ad\bigl(\Hol^{b^{j}_0}_{s}(\Phi(A))^{-1}\bigr)\xi\bigl(b^{j}_0\bigr)$}.
 Because \smash{$\Ad_{\mu(m)}\xi|_{\beta} = \xi|_{\beta}$},
 \[
 \eta_j\bigl(b^{j}_{\lambda_j}\bigr) = \Ad\bigl(\Hol^{b^{j}_0}_{b^{j}_{\lambda_j}}(\Phi(A))\bigr)\xi\bigl(b^{j}_0\bigr) = \xi\bigl(b^{j}_{\lambda_j}\bigr).
 \]
 Therefore, $\eta_j$ is well defined.
 By collecting $\eta_j$ for all $j =0,\dots,r$, we define $\eta \in \Omega^0(\partial \Sigma, \mathfrak{g})$.
 Because $(\eta-\xi)|_{\beta} = 0$, $\eta^{\#} = (\eta - \xi)^{\#} + \xi^{\#}$ is in $\ker\bigl(\sigma- \Phi^* \varpi\bigr)$.

We now prove $\iota\bigl(\eta^{\#}\bigr)\Phi^*\varpi=0$.
 For any $\alpha \in T_A \mathcal{M}(\Sigma)$, by Proposition~\ref{formula on loop group} and Lemma~\ref{lemma of calc of Hol},
 \begin{align*}
 \iota(\alpha)\iota\bigl(\eta^{\#}\bigr) \Phi^*\varpi &= \sum_{j=0}^{r} \iota(\alpha|_{V_j}) \iota\bigl({\eta|_{V_j}}^{\#}\bigr) \varpi_j\\
 &=-\sum_{j=0}^{r} \iota(\alpha|_{V_j}) {\rm d} \oint (A|_{V_j}, \eta|_{V_j}) - \frac{1}{2} \iota(\alpha)\Hol^*\bigl(\theta_{\bH}+ \overline{\theta}_{\bH}, \eta|_{\beta}\bigr) \\
 &=-\sum_{j=0}^{r} \oint (\alpha|_{V_j}, \eta|_{V_j}) - \iota(\alpha) \iota\bigl(\xi^{\#}\bigr) \omega\\
 &= -\sum_{j=0}^{r} \sum_{\lambda_j = 1}^{m_j} \int_{b^j_{\lambda_j}}^{b^j_{\lambda_{j+1}}} \bigl(\alpha|_{V_j}, \Ad\bigl(\Hol^{b^j_{\lambda_{j}}}_{s}(A|_{V_j})^{-1}\bigr)\xi\bigl(b^j_{\lambda_j}\bigr)\bigr){\rm d} s\\
 &=\sum_{j=0}^{r} \sum_{\lambda_j =1}^{m_j} \int _{b^j_{\lambda_j}}^{b^j_{\lambda_{j+1}}} \bigl(\Ad\bigl(\Hol^{b^j_{\lambda_j}}_{s}(A)\bigr)\alpha|_{V_j}, \xi\bigl(b^j_{\lambda_j}\bigr)\bigr) \\
 &= \sum_{j=0}^{r} \sum_{\lambda_j =1}^{m_j} \big( \iota(\alpha|_{V_j}){\big(\Hol^{b^j_{\lambda_j}}_{b^j_{\lambda_{j+1}}}\big)}^*\overline{\theta}, \xi\bigl(b^j_{\lambda_j}\bigr)\big) \\
 &= \frac{1}{2} \sum_{j=0}^{r} \sum_{\lambda_j=0}^{m_j} \big( \iota(\alpha|_{V_j}){\big(\Hol^{b^j_{\lambda_j}}_{b^j_{\lambda_{j+1}}}\big)}^*\overline{\theta}, \xi\bigl(b^j_{\lambda_j}\bigr) \big)\\
 & \qquad + \frac{1}{2} \sum_{j=0}^{r} \sum_{\lambda_j=0}^{m_j} \big( \iota(\alpha|_{V_j}){\big(\Hol^{b^j_{\lambda_j}}_{b^j_{\lambda_{j+1}}}\big)}^*\overline{\theta}, \Ad\bigl(\Hol^{b^j_{\lambda_j}}_{b^j_{\lambda_{j+1}}}\bigr)\xi\bigl(b^j_{\lambda_{j+1}}\bigr) \big) \\
 &= \iota(\alpha) \frac{1}{2} \Hol^*\bigl(\theta_{\bH}+ \overline{\theta}_{\bH}, \xi|_{\beta}\bigr) =0.
 \end{align*}
 Considering \smash{$\iota\bigl(\eta^{\#}\bigr) \Phi^*\varpi=0$} and \smash{$\eta^{\#} \in \ker\bigl(\sigma- \Phi^*\varpi\bigr)$}, we obtain $\smash{\iota\bigl(\eta^{\#}\bigr)\sigma= \iota\bigl(\eta^{\#}\bigr)}\bigl(\bigl(\sigma - \Phi^*\varpi\bigr)+\smash{\Phi^*\varpi}\bigr) =0$.
 Since $\sigma$ is weak non-degenerate, \smash{$\eta^{\#}$} is zero.
 Therefore, we conclude $\smash{{\xi|_{\beta}}^{\#} }= \smash{{(\xi-\eta)|_{\beta}}^{\#} + {\eta|_{\beta}}^{\#}}=0$.
\end{proof}

\begin{Remark}\quad
\begin{itemize}\itemsep=0pt
 \item[(1)] If the group $\Gamma$ is trivial, we can choose any finite subset of $\partial \Sigma$ as the set of base points.
 Then the space $M(\Sigma)$ is a~straightforward generalization of the space in~\cite[Theorem~9.1]{AMM}.
 \item[(2)] Although Theorem~\ref{q-Ham str on M(X)} follows directly from Theorem~\ref{equiv thm}, we include a~proof because the almost same argument applies to the case where $G$ is not simply connected.
 \end{itemize}
\end{Remark}

\subsection[Fixed points of M(X)]{Fixed points of $\boldsymbol{M(X)}$}\label{section_twisted_reps}
Taking fixed point objects of $(M(\Sigma),\omega,\mu)$, we obtain a~twisted $q$-Hamiltonian $\bigl(G^{\beta}\bigr)^{\Gamma}$-space $\bigl(M(\Sigma)^{\Gamma},\omega^{\Gamma},\mu^{\Gamma}\bigr)$.
We denote $ \underline{\Sigma} := \Sigma /\Gamma$ and $\beta_{\underline{\Sigma}} := \pi(\beta)$, where $\pi \colon \Sigma \to \underline{\Sigma}$ is the covering map.

Taking a~holonomy yields the following correspondence: $M(\Sigma) \simeq \Hom(\Pi_1(\Sigma,\beta),G)$.
This correspondence is $\Gamma$-equivariant for an action on $\Hom(\Pi_1(\Sigma,\beta),G)$
$(\varphi \cdot \rho)(\gamma) = \varphi \rho(\gamma\cdot\varphi)$, $
 \varphi \in \Gamma$, ${\rho \in \Hom(\Pi_1(\Sigma,\beta),G)}$, $ \gamma \in \Pi_1(\Sigma,\beta)$.
Therefore, we obtain a~bijection $M(\Sigma)^{\Gamma} \simeq \Hom(\Pi_1(\Sigma,\beta),G)^{\Gamma}$.

Now we want to consider the space $\Hom(\Pi_1(\Sigma,\beta),G)^{\Gamma}$ as the space on $\underline{\Sigma}$.
Let $G_{\underline{\Sigma}}$ be a~group bundle $(\Sigma \times G) /\Gamma$ on $\underline{\Sigma}$ with a~fiber $G$.
By fixing a~complete system of representatives $I$ of $\beta$ for~$\Gamma$-action,
we obtain the monodromy representation of $G_{\underline{\Sigma}}$ denoted by $\mon_{\underline{\Sigma},I} \colon \Pi_1(\underline{\Sigma},\beta_{\underline{\Sigma}}) \to \Gamma$ (cf.\ Appendix~\ref{Appendix B}).
We denote the space of representations twisted by the morphism $\mon_{\underline{\Sigma},I}$ by
\[
\Hom_{\mon_{\underline{\Sigma},I}}(\Pi_1(\underline{\Sigma},\beta_{\underline{\Sigma}}),G):= \bigr\{\underline{\rho} \in \Hom(\Pi_1(\underline{\Sigma},\beta_{\underline{\Sigma}}),\Gamma \ltimes G)\mid \pr \circ \underline{\rho} = \mon_{\underline{\Sigma},I}\bigr\}.
\]
Then, the covering map $\pi \colon \Sigma \to \underline{\Sigma}$ induces a~bijection
\[
\Hom_{\mon_{\underline{\Sigma},I}}(\Pi_1(\underline{\Sigma},\beta_{\underline{\Sigma}}),G) \simeq \Hom(\Pi_1(\Sigma,\beta),G)^{\Gamma}.
\]
Therefore, we obtain a~bijection
$M(\Sigma)^{\Gamma} \simeq \Hom_{\mon_{\underline{\Sigma},I}}(\Pi_1(\underline{\Sigma},\beta_{\underline{\Sigma}}),G)$.
By choosing a~system of free generators of $\Pi_1(\underline{\Sigma},\beta_{\underline{\Sigma}})$, the space $\Hom_{\mon_{\underline{\Sigma},I}}(\Pi_1(\underline{\Sigma},\beta_{\underline{\Sigma}}),G)$ is diffeomorphic to a~product of $G$-bitorsors.
Since $G$ is connected, this product of $G$-bitorsors is connected.
Therefore, the tuple~${\bigl(M(\Sigma)^{\Gamma}, \omega^{\Gamma}, \mu^{\Gamma}\bigr)}$ is a~twisted quasi-Hamiltonian \smash{$\bigl(G^{\beta}\bigr)^{\Gamma}$}-space.
Also \smash{$\Hom_{\mon_{\underline{\Sigma},I}}(\Pi_1(\underline{\Sigma},\beta_{\underline{\Sigma}}),G)$} carries a~twisted quasi-Hamiltonian $G^{\beta_{\underline{\Sigma}}}$-structure.
The above discussion is summarized in the following theorem.
\begin{Theorem}\label{q-Ham str on rep sp}
 The representation space $\Hom(\Pi_1(\Sigma,\beta),G)$ has a~structure of a~$\Gamma$-compatible quasi-Hamiltonian $G^{\beta}$-space.
 By taking the set of $\Gamma$-fixed points of $\Hom(\Pi_1(\Sigma,\beta),G)$, we obtain a~structure of a~twisted quasi-Hamiltonian $G^{\beta_{\underline{\Sigma}}}$-space on $\Hom_{\mon_{\underline{\Sigma}, I}}(\Pi_1(\underline{\Sigma},\beta_{\underline{\Sigma}}),G)$.
\end{Theorem}

\subsection[The 2-form on M(Sigma)]{The 2-form on $\boldsymbol{M(\Sigma)}$}
In this subsection, suppose that the base points on each boundary component $\beta_i \subset V_i$ are in one orbit for $\Gamma$-action on $\beta$.
Let $r_{\underline{\Sigma}}$ be a~non-negative integer such that $r_{\underline{\Sigma}}+1$ is the number of boundary components of $\underline{\Sigma}$ and let $g_{\underline{\Sigma}}$ be the genus of $\underline{\Sigma}$.
Then $\beta_{\underline{\Sigma}}$ has only one point on each boundary component of $\underline{\Sigma}$.
We denote $\beta_{\underline{\Sigma}} = \big\{b^0,b^1,\dots,b^{r_{\underline{\Sigma}}}\big\}$. Let $Q_{\underline{\Sigma}}$ denote a~quiver which generates free groupoid $\Pi_1(\underline{\Sigma},\beta_{\underline{\Sigma}})$,
$Q_{\underline{\Sigma}} := \big\{a^k,b^k,\partial^j,\gamma^j \mid k=1,\dots,g_{\underline{\Sigma}}, \, j=1,\dots,r_{\underline{\Sigma}} \big\}$,
where $a^k$ and $b^k$ are $a$-cycles and $b$-cycles and a~$\partial^j$ is a~loop along a~boundary component based at~$b^j$ and a~$\gamma^j$ is a~path from $b^0$ to $b^j$.
Set $\beta_j := \pi^{-1}\bigl(b^j\bigr)$ and we label $\beta_0 = \big\{b^0_1,\dots, b^0_m\big\}$.
Furthermore, let~\smash{$\gamma^j_\lambda$} be the unique lifting of \smash{$\gamma^j$} starting at \smash{$b^0_\lambda$}.
By using \smash{$\gamma^j_\lambda$}, we set \smash{$\beta^{j}_{\lambda} := t\bigl(\gamma^{j}_{\lambda}\bigr)$}, where~$\smash{t\bigl(\gamma^{j}_{\lambda}\bigr)}$ is the target of \smash{$\gamma^{j}_{\lambda}$}.
Set a~quiver, $Q_{\Sigma} := \big\{a^k_\lambda,b^k_\lambda,\partial^j_\lambda,\gamma^j_\lambda\big\}$,
the fundamental groupoid~$\Pi_1(\Sigma,\beta)$ is the free groupoid (cf.\ \cite{mac1998categories}) of $Q_{\Sigma}$.

We define four $\Gamma$-compatible $G^m$-bitorsors (see Example~\ref{exam_bitorsor}) by
\begin{align*}
 & \bC^{j} := \big\{C^{(j)} = \bigl(\rho\bigl(\gamma^{j}_{\lambda}\bigr)\bigr)_{\lambda=1}^{m}\mid \rho \in \Hom(\Pi_1(\Sigma,\beta),G)\big\}, \\
 & \bH^{j}:= \big\{h^{(j)} = \bigl(\rho\bigl(\partial^{j}_{\lambda}\bigr)\bigr)_{\lambda=1}^{m}\mid \rho \in \Hom(\Pi_1(\Sigma,\beta),G)\big\}, \\
& \bA^{k} := \big\{A^{(k)} = \bigl(\rho\bigl(a^{k}_{\lambda}\bigr)\bigr)_{\lambda=1}^{m}\mid \rho \in \Hom(\Pi_1(\Sigma,\beta),G)\big\}, \\
& \bB^{k} := \big\{B^{(k)} = \bigl(\rho\bigl(b^{k}_{\lambda}\bigr)\bigr)_{\lambda=1}^{m}\mid \rho \in \Hom(\Pi_1(\Sigma,\beta),G)\big\}.
\end{align*}

We define permutations $t^{j}_{\partial}$, $t^{k}_{a}$, $t^{k}_{b}$ of $\{1,\dots,m\}$ by
\begin{align*}
 s\bigl(\partial^{j}_{t^{j}_{\partial}(\lambda)}\bigr) &= t\bigl(\partial^{j}_{\lambda}\bigr), \qquad
 s\bigl(a^{k}_{t^{k}_{a}(\lambda)}\bigr) = t\bigl(a^{k}_{\lambda}\bigr), \qquad
 s\bigl(b^{k}_{t^{k}_{b}(\lambda)}\bigr) = t\bigl(b^{k}_{\lambda}\bigr).
\end{align*}
Also we define left and right $G^{m}$-actions on each space by
\begin{alignat*}{3}
& \bigl(g \cdot C^{(j)}\bigr)_{\lambda} := g_{\lambda}C^{(j)}_{\lambda},\qquad&& \bigl(C^{(j)} \cdot g\bigr)_{\lambda} := C^{(j)}_{\lambda}g_{\lambda},& \\
& \bigl(g \cdot h^{(j)}\bigr)_{\lambda} := g_{t^{j}_{\partial}(\lambda)}h^{(j)}_{\lambda},\qquad&& \bigl(h^{(j)}\cdot g\bigr)_{\lambda} := h^{(j)}_{\lambda} g_{\lambda}, &\\
& \bigl(g \cdot A^{(k)}\bigr)_{\lambda} \colon\ g_{t^{k}_{a}(\lambda)}A^{(k)}_{\lambda},\qquad &&\bigl(A^{(k)}\cdot g\bigr)_{\lambda} \colon\ A^{(k)}_{\lambda} g_{\lambda},&\\
& \bigl(g \cdot B^{(k)}\bigr)_{\lambda} := g_{t^{k}_{b}(\lambda)}B^{(k)}_{\lambda},\qquad&& \bigl(B^{(k)}\cdot g\bigr)_{\lambda} := B^{(k)}_{\lambda} g_{\lambda}.&
\end{alignat*}

So, we can construct the double $\textbf{D}^{(j)} := \textbf{D}\bigl(\bC^j,\bH^j\bigr)$, with coordinates in Remark~\ref{the double with different coordinates}, and the fused double \smash{$\bD^{(k)} := \bD\bigl(\bA^k,\bB^k\bigr)$}.
Fusing the \smash{$\textbf{D}^{(j)}$}'s with respect to the first component $G^m$ of the $G^m \times G^m$-action together with all $\bD^{(k)}$, we obtain the $\Gamma$-compatible quasi-Hamiltonian $G^{\beta}$-space
\[
\bigfusion_{j} \textbf{D}^{(j)} \circledast \bigfusion_{k} \bD^{(k)}.
\]
Since $\Pi_1(\Sigma,\beta)$ is the free groupoid of $Q_{\Sigma}$, we have a~structure of a~$\Gamma$-compatible quasi-Hamiltoni\-an $G^{\beta}$-space on $\Hom(\Pi_1(\Sigma,\beta),G)$, via the canonical bijection
\[
\Hom(\Pi_1(\Sigma,\beta),G) = \bigfusion_{j} \textbf{D}^{(j)} \circledast \bigfusion_{k} \bD^{(k)}.
\]
This structure of $\Gamma$-compatible quasi-Hamiltonian $G^{\beta}$-space coincides with the structure coming from~$M(\Sigma)$.
This is analogous to~\cite[Theorem 9.3]{AMM}.

\begin{Theorem}
 $(M(\Sigma),\omega,\mu)$ is canonically isomorphic to the fusion product
 \[ \bigfusion_{j} \textbf{D}^{(j)} \circledast \bigfusion_{k} \bD^{(k)}\]
 as $\Gamma$-compatible quasi-Hamiltonian $G^{\beta}$-spaces.
\end{Theorem}

\begin{proof}
 The idea of this proof is based on~\cite[Theorem 9.3]{AMM}.
 Let $P$ denote the polyhedron obtained by cutting $\underline{\Sigma}$ along the paths $\gamma^j$, $a^j$, $b^j$.
 The boundary $\partial P$ consists of $3r_{\underline{\Sigma}} + 4g_{\underline{\Sigma}}+1$ segments
\smash{$
 \partial P = \bigl(\gamma^1\bigr)^{-1}\partial^1 \bigl(\gamma^1\bigr)\cdots\bigl(\gamma^{r_{\underline{\Sigma}}}\bigr)^{-1}\partial^{r_{\underline{\Sigma}}}\bigl(\gamma^{r_{\underline{\Sigma}}}\bigr)
 \bigl[a^1,b^1\bigr]\cdots[a^{g_{\underline{\Sigma}}},b^{g_{\underline{\Sigma}}}]\partial^0$}.

 Let $\partial P_\lambda$ be the lifting of $\partial P $ starting at $b^0_\lambda$.
 Then the loop $\partial P_\lambda$ is homotopic to a~constant path.
 Let $P_\lambda$ be a~lifting of $P$ such that its boundary is $\partial P_\lambda$.
 Then we write $\Sigma = \bigcup_{\lambda=1}^{m} P_\lambda $,
 where each $P_\lambda $ intersects the others only along their boundary segments.
 We define a~$2$-form on~$\mathcal{A}(\Sigma)$ by $\sigma_\lambda := \frac{1}{2} \int_{P_\lambda} (\delta A, \delta A)$.
 Then Atiyah--Bott $2$-form can be written by $\sigma = \sum_{\lambda=1}^{m} \sigma_\lambda$.
 We obtain the following equation of $2$-forms on $\mathcal{A}(\Sigma)$ using~\cite[Lemma 9.4]{AMM} for each $\sigma_{\lambda}$
 \begin{align*}
 \sigma-\sum_{i=0}^{r_{\Sigma}} R_{i}^* \varpi_i ={} & \frac{1}{2}\sum_{j=1}^{r_{\underline{\Sigma}}} \sum_{\lambda} \big({C^{(j)}_{t^{j}_{\partial}(\lambda)}}^*\overline{\theta}, \Ad_{h^{(j)}_{\lambda}} {C^{(j)}_{\lambda}}^* \overline{\theta}\big)
 + \frac{1}{2}\sum_{j=1}^{r_{\underline{\Sigma}}} \sum_{\lambda} \big({C^{(j)}_{t^{j}_{\partial}(\lambda)}}^*\overline{\theta}, {h^{(j)}_{t^{j}_{\partial}(\lambda)}}^*\theta + {h^{(j)}_{\lambda}}^*\overline{\theta}\big) \\
 &- \frac{1}{2}\sum_{k=1}^{g_{\underline{\Sigma}}} \sum_{\lambda} \big( {a^{(k)}_{t^{k}_b(\lambda)}}^*\theta, {b^{(k)}_{\lambda}}^*\overline{\theta} \big) -\frac{1}{2}\sum_{k=1}^{g_{\underline{\Sigma}}} \sum_{\lambda} \big( {a^{(k)}_{\lambda}}^*\overline{\theta}, {b^{(k)}_{t^{(k)}_{a}(\lambda)}}^*\theta \big) \\
 &- \frac{1}{2}\sum_{k=1}^{g_{\underline{\Sigma}}} \sum_{\lambda} \big( \bigl({a^{(k)}_{t^{k}_b(\lambda)}}b^{(k)}_{\lambda}\bigr)^* \theta, \bigl({a^{(k)}_{\lambda}}^{-1}{b^{(k)}_{t^{(k)}_{a}(\lambda)}}^{-1}\bigr)^* \overline{\theta}\big) + \omega_{\mathrm{fus}},
 \end{align*}
 where $\omega_{\mathrm{fus}}$ is the pull back of extra terms coming from fusion.
 Moreover, the right-hand side is the quasi-Hamiltonian $2$-form on $\Hom(\Pi_1(\Sigma,\beta),G)$.
\end{proof}

\appendix

\section{Generalizations of the twisted double}\label{Appendix A}

\begin{figure}[t]\centering
\includegraphics[scale=0.7]{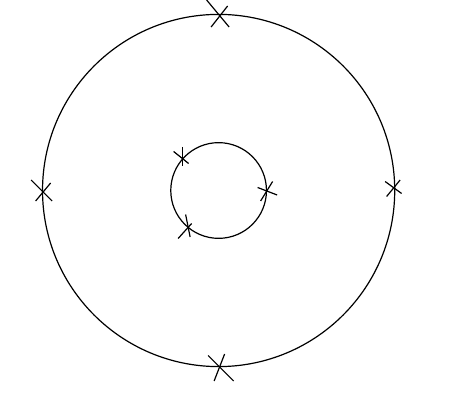}
\end{figure}

By considering the moduli space of an annulus $\Sigma$ with base points $\beta = \beta_0 \cup \beta_{\infty}$ as above,
we obtain the twisted quasi-Hamiltonian space which generalizes the twisted double, and we describe its quasi-Hamiltonian structure explicitly.

We fix a~cyclic order $\beta_0 = \big\{ b^0_1,\dots,b^0_{m_0}\big\}$ and $\beta_{\infty} = \{b^{\infty}_1,\dots,b^{\infty}_{m_\infty}\}$.
We denote paths by
$\gamma \colon b^0_1 \to b^{\infty}_1$,
$\partial^{\infty}_i \colon b^{\infty}_i \to b^{\infty}_{i+1}$,
$\partial^{0}_i \colon b^0_i \to b^0_{i+1}$.
And for any $\rho \in \Hom(\Pi_1(\Sigma,\beta),G)$, we denote the corresponding values by
 $C := \rho(\gamma)$,
 $h_i := \rho\bigl(\partial ^{\infty}_i\bigr)$,
 $h^0_i := \rho\bigl(\partial^0_i\bigr)$.

Then, the space $\Hom(\Pi_1(\Sigma,\beta),G)$ is identified with
\[
\bigl\{\bigl(C,(h_i)_{i=1}^{m_\infty}, \bigl(h^0_i\bigr)_{i=1}^{m_0}\bigr) \in G \times G^{m_\infty} \times G^{m_0}\mid
C^{-1}h_{m_\infty}\cdots h_1 C h^0_{m_0} \cdots h^0_{1}=1\bigr\}.\]

Then the moment map $\mu$ and the $2$-form $\omega$ are given as follows:
\begin{gather*}
 \mu\bigl(C,h,h^0\bigr) = \bigl(h^{-1},\bigl(h^{0}\bigr)^{-1}\bigr),\\
 \omega = \frac{1}{2}\bigl(C^*\overline{\theta}, \Ad_{h}C^*\overline{\theta}\bigr) + \frac{1}{2}\bigl(C^*\overline{\theta},h^*\overline{\theta}+h^*{\theta}\bigr) + \frac{1}{2}\sum_{i=1}^{m_\infty}\bigl(\bigl(k^\infty_i\bigr)^*\overline{\theta}, h_i^*\theta\bigr)\\
 \phantom{ \omega = }{}
 + \frac{1}{2}\sum_{i=1}^{m_0}\bigl(\bigl(k^0_i\bigr)^*\overline{\theta}, \bigl(h^0_i\bigr)^*\theta\bigr),
\end{gather*}
where $k^{\infty}_i$ and $k^0_i$ are defined by $k^{\infty}_i := h_{i-1}\cdots h_1$ and $k^0_i := h^0_{i-1}\cdots h^0_1$, respectively.

We denote the $q$-Hamiltonian space $\Hom(\Pi_1(\Sigma,\beta),G)$ by $_{m_{\infty}} \textbf{D}_{m_0}$.
Then, if $m_\infty = m_0$, the space $_{m_\infty} \textbf{D}_{m_0}$ is the double $\textbf{D}\bigl(G^{m_0},G^{m_0}\bigr)$.
It is easy to check that the space $_{m_\infty}\textbf{D}_{m_0}$ is obtained by gluing $_{m_{\infty}} \textbf{D}_{1}$ and $_{1} \textbf{D}_{m_0}$.

\section{Twisted group bundles and monodromy}\label{Appendix B}

\subsection{Setting}

Let $\Sigma$ be an oriented compact 2-manifold with boundary.
Let $\Gamma$ be a~finite group.
We consider a~left $\Gamma$-action on $G$ and a~free right $\Gamma$-action on $\Sigma$.
We define
 $ \underline{\Sigma} := \Sigma/\Gamma$,
 trivial group bundle $G_{\Sigma} := \Sigma \times G$ on $\Sigma$,
a~group bundle $G_{\underline{\Sigma}}:= (\Sigma \times G)/\Gamma$ on $\underline{\Sigma}$.

Let $\beta_{\underline{\Sigma}} $ be a~finite subset of $\underline{\Sigma}$, and let $\beta_{\Sigma}$ be a~pre-image of $\beta_{\underline{\Sigma}}$ for the quotient map $q \colon \Sigma \to \underline{\Sigma}$.
We denote the fundamental groupoid by $\Pi_{\Sigma} := \Pi_1(\Sigma,\beta_{\Sigma})$ and $\Pi_{\underline{\Sigma}} := \Pi_1(\underline{\Sigma},\beta_{\underline{\Sigma}})$.

\subsection[The monodromy of G\_Sigma]{The monodromy of $\boldsymbol{G_{\underline{\Sigma}}}$}
The trivial group bundle $G_{\Sigma}$ has the trivial flat Ehresmann connection $\mathcal{H}_{\Sigma}$. We denote its monodromy by $\mon_{\Sigma}$.
Since $G_{\Sigma}$ is trivial as a~group bundle, under the natural identification~${G_b \simeq G}$ for all $b \in \beta_{\Sigma}$, we obtain the trivial monodromy representation
$
\mon_{\Sigma, \beta} \colon\Pi_{\Sigma} \to \Aut(G), \gamma \mapsto 1$.

Since $\mathcal{H}_{\Sigma}$ is $\Gamma$-invariant under the diagonal action of $\Gamma$ on $G_{\Sigma}$, $\mathcal{H}_{\Sigma}$ induces a~flat Ehresmann connection $\mathcal{H}_{\underline{\Sigma}}$ on $\underline{\Sigma}$.
We denote its monodromy by $\mon_{\underline{\Sigma}}$.
For any $\gamma \in \Pi_{\underline{\Sigma}}$, $\mon_{\underline{\Sigma}}(\gamma) \in \Isom(G_{s(\gamma)},G_{t(\gamma)})$.

Fix a~complete system of representatives $I \subset \beta_{\Sigma}$ for the $\Gamma$-action.
We construct a~$\Gamma$-valued representation of $\mon_{\underline{\Sigma}}$.
For any $\gamma \in \Pi_{\underline{\Sigma}}$, $i(s(\gamma)) \in I$ denotes a~representative of~$q^{-1}(s(\gamma))$, and~${i(t(\gamma)) \in I}$ denotes a~representative of $q^{-1}(t(\gamma))$.
Let $\tilde{\gamma}$ be a~lift of $\gamma \in \Pi_{\underline{\Sigma}}$ starting at~$i(s(\gamma))$,
then there exists a~unique element $\varphi_I(\gamma) \in \Gamma$ such that
$t(\tilde{\gamma}) = i(t(\gamma))\cdot \varphi_I(\gamma)$.
Set ${\mon_{\underline{\Sigma},I}(\gamma) := \varphi_I(\gamma)}$,
we obtain the monodromy representation $\mon_{\underline{\Sigma},I} \colon \Pi_{\underline{\Sigma}} \to \Gamma$.

\subsection[The G-valued representation of Pi\_Sigma]{The $\boldsymbol{G}$-valued representation of $\boldsymbol{\Pi_{\underline{\Sigma}}}$}

We denote
\[
\Hom_{\mon_{\underline{\Sigma},I}}(\Pi_{\underline{\Sigma}},G) := \{\underline{\rho} \in \Hom(\Pi_{\underline{\Sigma}}, \Gamma \ltimes G)\mid \pr_{\Gamma} \circ \underline{\rho} = \mon_{\underline{\Sigma},I}\}.
\]
In this appendix, we construct a~bijection
\[
\Hom(\Pi_{\Sigma},G)^{\Gamma} \to \Hom_{\mon_{\underline{\Sigma},I}}(\Pi_{\underline{\Sigma}},G), \qquad \rho \mapsto \underline{\rho}_I.
\]
 A~$\Gamma$-action on $\Hom(\Pi_{\Sigma},G)$ is defined by
$
(\varphi \cdot \rho)(\tilde{\gamma}) = \varphi \rho(\tilde{\gamma}\cdot\varphi) $, $ \rho \in \Hom(\Pi_{\Sigma},G)$, $ \varphi \in \Gamma$, $ \tilde{\gamma} \in \Pi_{\Sigma}$.
For any $\rho \in \Hom(\Pi_{\Sigma},G)^{\Gamma}$ and $\gamma \in \Pi_{\underline{\Sigma}}$, we denote the lifting of $\gamma$ starting at $i(s(\gamma))$ by $\tilde{\gamma}$.
We define $\underline{\rho}_{I}(\gamma) := (\mon_{\underline{\Sigma},I}(\gamma),\rho(\tilde{\gamma}))$.

\begin{Proposition}
 \smash{$\underline{\rho}_I \colon \Pi_{\underline{\Sigma}} \to \Gamma \ltimes G$} is a~groupoid homomorphism.
\end{Proposition}

\begin{proof}
 For any $\gamma_1,\gamma_2 \in \Pi_{\underline{\Sigma}}$ satisfying $s(\gamma_1)=t(\gamma_2)$, \smash{$\widetilde{(\gamma_1 \gamma_2)} = (\tilde{\gamma_1}\cdot \mon_{Y,I}(\gamma_2))\tilde{\gamma_2}$}.
 Therefore,
 \begin{align*}
 \underline{\rho}_I(\gamma_1\gamma_2) &= (\mon_{Y,I}(\gamma_1 \gamma_2), \rho(\widetilde{\gamma_1 \gamma_2}))
 = (\mon_{Y,I}(\gamma_1) \mon_{Y,I}(\gamma_2), \rho(\tilde{\gamma_1}\cdot\mon_{Y,I}(\gamma_2))\rho(\tilde{\gamma_2})) \\
 &= (\mon_{Y,I}(\gamma_1) \mon_{Y,I}(\gamma_2), (\rho(\tilde{\gamma_1}) \cdot \mon_{Y,I}(\gamma_2)) \rho(\tilde{\gamma_2}))
 = \underline{\rho}_{I}(\gamma_1) \underline{\rho}_{I}(\gamma_2).\tag*{\qed}
 \end{align*} \renewcommand{\qed}{}
\end{proof}

It is easy to see that the above correspondence is bijective.

\section[A C\string^infty structure of fixed point sets]{A $\boldsymbol{C^{\infty}}$ structure of fixed point sets}\label{Appendix C}

In this appendix, we show that each connected component of the fixed point set $N^{\Gamma}$ for a~$C^{\infty}$-action of a~finite group $\Gamma$ on a~Banach manifold $N$ is a~Banach submanifold of $N$.

Fix a~point $x \in N^{\Gamma}$, a~$\Gamma$-action on the tangent space $T_xN$ is induced by the $\Gamma$-action on $N$.
We denote the averaging by $\ave_x \colon T_xN \to (T_xN)^{\Gamma}$, $ X \mapsto \frac{1}{|\Gamma|} \sum_{\varphi \in \Gamma} \varphi \cdot X$.
Because the $\Gamma$-action on $T_xN$ is top-linear, the averaging is also top-linear.

\begin{Lemma}
 $\ker (\ave_x)$ and $(T_xN)^{\Gamma}$ are closed subspaces of $T_xN$.
 Furthermore, a~sequence
 \[0 \to \ker(\ave_x) \hookrightarrow T_xN \twoheadrightarrow (T_xN)^{\Gamma} \to 0\]
 is split.
\end{Lemma}

\begin{proof}
 Since $(T_xN)^{\Gamma}$ is an intersection of closed subspaces \smash{$(T_xN)^{\Gamma} = \bigcap_{\varphi \in \Gamma} \ker(\varphi - 1)$},
 the space $(T_xN)^{\Gamma}$ is closed.
 Because $\ave_x$ is top-linear, the kernel $\ker(\ave_x)$ is a~closed subspace.
 For any $X \in T_xN$,
 \begin{align*}
 (\ave_x)^2(X) &= \ave_x\bigg(\frac{1}{|\Gamma|}\sum_{\varphi \in \Gamma}\varphi\cdot X \bigg)
 = \frac{1}{|\Gamma|}\sum_{\varphi \in \Gamma} \ave_x(\varphi \cdot X)
 = \frac{1}{|\Gamma|} \sum_{\varphi \in \Gamma} \varphi \cdot X
 = \ave_x(X).
 \end{align*}
 Therefore, we have $T_xN = \ker(\ave_x) \oplus (T_xN)^{\Gamma}$.
\end{proof}

Suppose that $N^{\Gamma}$ is connected. Fix an element $x \in N^{\Gamma}$.
We show that there exists a~good chart around $x \in N$.
By Bochner's linearization theorem for compact Lie group actions on Banach manifolds~\cite[Proposition~3.1.4]{diez2019normal}, there exists a~$\Gamma$-invariant open neighborhood $U$ of~${x \in N^{\Gamma}}$
and a~$\Gamma$-equivariant diffeomorphism $\chi \colon U \stackrel{\sim}{\rightarrow} V \subset T_xN $ which satisfies ${\rm d}_x\chi =1_{T_xN}$.
Since $T_xN = \ker(\ave_x) \oplus (T_xN)^{\Gamma}$, the above neighborhood $U$ can be written as $ U= \chi^{-1}\bigl(V^{\prime}\times V^{\Gamma}\bigr)$
where $V^{\prime}$ \big(resp.\ $V^{\Gamma}$\big) is a~sufficiently small open neighborhood of $0$ in $\ker(\ave_x)$ \big(resp.\ $(T_xN)^{\Gamma}$\big).
Because $\chi$ is $\Gamma$-equivariant, we have $N^{\Gamma} \cap U \simeq V^{\Gamma} = V \cap (T_xN)^{\Gamma}$.
Therefore, each connected component of $N^{\Gamma}$ is a~Banach submanifold of $N$.

\subsection*{Acknowledgments}
The author would like to thank Professor Daisuke Yamakawa for his valuable advice.
We are also grateful to the referees for their valuable suggestions and comments, which have helped us improve the paper considerably.

\pdfbookmark[1]{References}{ref}
 \LastPageEnding

\end{document}